\documentclass[a4paper,12pt]{amsart}

\usepackage{mathrsfs}

\usepackage{amssymb}
\usepackage{latexsym}
\usepackage{amsfonts}
\usepackage{amsmath}
\usepackage{eucal}
\usepackage{bm}
\usepackage{bbm}
\usepackage{graphicx}
\usepackage[english]{varioref}
\usepackage[nice]{nicefrac}
\usepackage[all]{xy}
\usepackage{amsthm}
\usepackage[margin=1.25in]{geometry}

\newcommand{\supp}{\text {\rm supp}}

\def\i{^{-1}}
\def\ge{\geqslant}
\def\le{\leqslant}
\def\<{\langle}
\def\>{\rangle}

\def\ba{\textbf{a}}

\def\Adm{{\rm{Adm}}}
\def\ad{{\rm{ad}}}

\def\a{\alpha}
\def\b{\beta}
\def\g{\gamma}
\def\G{\Gamma}
\def\d{\delta}

\def\e{\epsilon}

\def\o{\omega}

\def\s{\sigma}

\def\th{\theta}

\def\l{\lambda}
\def\z{\zeta}

\def\tPhi{\tilde \Phi}

\def\ZZ{\mathbb Z}
\def\AA{\mathbb A}

\def\NN{\mathbb N}
\def\QQ{\mathbb Q}
\def\JJ{\mathbb J}
\def\FF{\mathbb F}
\def\RR{\mathbb R}

\def\PP{\mathbb P}
\def\SS{\mathbb S}
\def\kk{\bold{k}}

\def\ca{\mathcal A}

\def\cc{\mathcal C}

\def\co{\mathcal O}

\def\car{\mathcal R}

\def\car{\mathcal R}

\def\tW{\tilde W}
\def\tw{\tilde w}

\def\ba{\bold{a}}

\def\ad{{\rm ad}}

\def\dist{\text{dist}}

\theoremstyle{plain}
\newtheorem{thm}{Theorem}[section]
\newtheorem{conj}{Conjecture}[section]

\newtheorem*{thm*}{Theorem}
 \newtheorem{prop}[thm]{Proposition}
 \newtheorem{lem}[thm]{Lemma}
 \newtheorem{cor}[thm]{Corollary}

\theoremstyle{definition}

\theoremstyle{remark}

\newtheorem*{claim*}{Claim}

\begin{document}

\title[]{Connected components of closed affine Deligne-Lusztig varieties}
\keywords{Affine Deligne-Lusztig varieties; connected components.}
\subjclass[2010]{20G25, 14G35}

\author{Ling Chen}
\address{School of Mathematical Sciences, University of Chinese Academy of Sciences, Beijing 100049, China}
\email{chenling2013@ucas.ac.cn}

\author{Sian Nie}
\address{Institute of Mathematics, Academy of Mathematics and Systems Science, Chinese Academy of Sciences, 100190, Beijing, China}
\email{niesian@amss.ac.cn}

\thanks{L.C. is supported in part by NSFC grant 11401559 and UCAS grant Y55202HY00. S.N. is supported in part by QYZDB-SSW-SYS007 and NSFC grant (No. 11501547 and No. 11621061).}

\begin{abstract}
For split reductive algebraic groups, we determine the connected components of closed affine Deligne-Lusztig varieties of arbitrary parahoric level.
\end{abstract}

\maketitle

\section*{Introduction} \label{first}
\subsection{} The motivation of this paper comes from the study of reductions of PEL type Shimura varieties at a prime number $p$. A Shimura varity of PEL type can be viewed as a moduli space of abelian varieties with additional structures such as polarization, endomorphisms, level structures and so on. To each such abelian variety we can attach its $p$-divisible group, which inherits corresponding additional structures. Thus the special fiber of the moduli space at $p$ decomposes into finitely many locally closed subspaces called Newton strata, which are parameterized by the isogeny classes of these $p$-divisible groups (with additional structures).

By the uniformization theorem of Rapoport and Zink \cite{RZ}, one can describe the Newton strata in terms of so-called Rapoport-Zink spaces. Their geometric properties play an important role in the study of Shimura varieties. Via Dieudonn\'{e} theory, the set of geometric points of a Rapoport-Zink space can be identified with certain union of affine Deligne-Lusztig varieties. The main purpose of this paper is to study the connected components of such unions of affine Deligne-Lusztig varieties.

Let $\FF_q$ be a finite field of $q$ elements and let $\kk$ be its algebraic closure. Let $F$ be a finite extension of $\QQ_p$ with residue class field $\FF_q$ or $F=\FF_q((t))$ be the field of Laurent series over $\FF_q$. Denote by $L$ the completion of a maximal unramified extension of $F$. We fix a uniformizer $t$ of $F$.

Let $G$ be an unramified reductive group over $F$. Let $\s$ be the Frobenius automorphism of $L / F$. We also denote by $\s$ the induced automorphism on $G(L)$.

Fix a maximal torus $T$ and a Borel subgroup $B \supseteq T$ over $F$. The Iwahori-Weyl group is defined by $\tW=N(L) / T(L)_0$, where $N$ is the normalizer of $T$ in $G$ and $T(L)_0$ is the unique maximal compact subgroup of $T(L)$. Fix a $\s$-invariant Iwahori subgroup $I \subseteq G(L)$ containing $T(L)_0$. We have the Iwahori-Bruhat decomposition $G(L)=\sqcup_{x \in \tW} I x I$. Let $P \supseteq I$ be a $\s$-stable parahoric subgroup. For $b \in G(L)$ and $x \in \tW$, the attached affine Deligne-Lusztig variety is define by $$X_x(b)_P=\{g P \in G(L); g\i b \s(g) \in P x P\} / P.$$ If $F$ is of equal characteristic, this is the set of $\kk$-points of a locally closed subscheme in the partial affine flag variety $G(L) / P$. If $F$ is of mixed characteristic, it is the set of $\kk$-points of a locally closed perfect subscheme of the $p$-adic partial flag variety in the sense of \cite{BS} and \cite{Zhu}.

In the theory of local Shimura varieties (or Rapoport-Zink spaces), it is natural to consider the following union of affine Deligne-Lusztig varieties: $$X(\l, b)_P = \cup_{x \in \Adm(\l)} X_x(b)_P,$$ where $\l$ is a geometric cocharacter of $T$ and $\Adm(\l) \subseteq \tW$ denotes the $\l$-admissible set (see \S\ref{admissible} for notation). When the data $(G, \l, P)$ arises from a (PEL type) Shimura variety (in particular, $F$ is of mixed characteristic), $X(\mu, b)_P$ is the set of $\kk$-valued points of the corresponding Rapoport-Zink space (see \cite{R}).

We are interested in the set $\pi_0(X(\l, b)_P)$ of connected components of $X(\mu, b)_P$. When $P$ is hyperspecial, $\pi_0(X(\mu, b)_P)$ is determined by Viehmann \cite{Vie} if $G$ is split, by Chen-Kisin-Viehmann \cite{CKV} if $\mu$ is minuscule and by the second named author \cite{N2} if $\mu$ is non-minuscule. Recently, He and Zhou \cite{HZ} determined $\pi_0(X(\mu, b)_P)$ when $b$ is basic \footnote{Actually, they obtained such result for an arbitrary connected reductive group over $F$.}. Moreover, in the residue split case, they obtained a description of $\pi_0(X(\mu, b)_P)$ (with $b$ arbitrary) in terms of straight elements and their associated Levi subgroups (see \cite[Theorem 0.2]{HZ}). These results play an essential role in describing connected components of unramified Rapoport-Zink spaces (see \cite{Ch}, \cite{CKV}) and in verifying the Langlands-Rapoport conjecture for mod-$p$ points on Shimura varieties (see \cite{Ki}, \cite{Zh}).

\subsection{} \label{step} As indicated in the pioneer work \cite{Vie} and \cite{CKV}, we can determine $\pi_0(X(\l, b)_P)$ in the following three steps.

The first step is reduction to adjoint and simple groups. Let $G_{\ad}$ be the adjoint group of $G$, and denote by $b_{\ad}$, $\l_{\ad}$ and $P_{\ad}$ the images of $b$, $\l$ and $P$ in $G_{\ad}(L)$ respectively. Then $\pi_0(X^G(\l, b)_P)$ can be computed from $\pi_0(X^{G_{\ad}}(\l_{\ad}, b_{\ad})_{P_{\ad}})$ via the following Cartesian diagram
\[
\xymatrix{
  \pi_0(X^G(\l, b)_P) \ar[d]_{\eta_G} \ar[r] & \pi_0(X^{G_{\ad}}(\l_{\ad}, b_{\ad})_{P_{\ad}}) \ar[d]^{\eta_{G_{\ad}}} \\
  \pi_1(G)_{\G} \ar[r] & \pi_1(G_{\ad})_{\G}.  }
\]
Once $G$ is adjoint, it is a product of simple groups over $F$. So we can assume $G$ is simple and adjoint.

The second step is reduction to Hodge-Newton indecomposable case. If the pair $(\l, b)$ is Hodge-Newton decomposable, then \cite[Theorem 3.8]{GHN2} says that $X^G(\l, b)_P$ is a disjoint union of closed subvarieties of the form $X^{G'}(\l', b')_{P'}$, where $G' \supseteq T$ is some proper Levi subgroup of $G$ over $F$. Thus, by induction on the semisimple rank of $G$, it suffices to consider the Hodge-Newton indecomposable case.

The third step is explicit computation in Hodge-Newton irreducible case. Suppose $(\l, b)$ is Hodge-Newton indecomposable and $G_{\ad}$ is simple over $F$. By \cite[Theorem 2.5.6]{CKV}, either $b$ is $\s$-conjugate to $t^\l$ with $\l$ central in $G$ or $(\l, b)$ is Hodge-Newton irreducible (see Section \ref{proof-main} for notation). In the former case, we have $$X(\mu, b)_P \cong \JJ_b / (\JJ_b \cap P),$$ which is discrete. Here $\JJ_b=\{g \in G(L); g\i b \s(g)=b\}$ denotes the $\s$-centralizer of $b$. Therefore, it remains to consider the Hodge-Newton irreducible case.
\begin{conj} \label{main}
If $(\l, b)$ is Hodge-Newton irreducible, then the natural projection $\eta_G: G(L) / P \to \pi_1(G)$ induces a bijection $$\pi_0(X(\l, b)_P) \cong (\s-1)\i(\eta_G(t^\l)-\eta_G(b)) \subseteq \pi_1(G).$$
\end{conj}

\subsection{} The main result of this paper is the following.
\begin{thm} \label{intro}
Conjecture \ref{main} is true if $G$ splits over $F$.
\end{thm}

Thanks to \cite{He2}, the natural projection $X(\l, b):=X(\l, b)_I \to X(\l, b)_P$ is surjective. Thus, to prove Theorem \ref{intro}, it suffices to consider the Iwahori case. To this end, we adopt the strategies of \cite{CKV} and \cite{HZ}. The starting point is Proposition \ref{point}, which says that each connected component of $X(\l, b)$ intersects with some $X_x(b)$, where $x$ ranges over the set $\Adm(\l)_{str}$ of straight elements (see \S \ref{admissible} for notation) contained in $\Adm(\l)$. Since $X_x(b) \cong \pi_1(G)^\s$, we deduce that $$\eta_G(X(\l, b))= (\s-1)\i(\eta_G(t^\l)-\eta_G(b)) \cong \pi_1(G)^\s.$$ Then it remains to prove that $Q' \in X_{x'}(b)$ and $Q'' \in X_{x''}(b)$ with $x', x'' \in \Adm(\l)_{str}$ are connected in $X(\l, b)$ if $\eta_G(Q')=\eta_G(Q'')$. We verify this statement by constructing curves in $X(\l, b)$ connecting $Q'$ and $Q''$. This completes the proof of Theorem \ref{intro}.

The paper is organized as follows. In \S \ref{pre}, we recall basic notations and lemmas. In \S\ref{proof-main}, we introduce the notion of permissible elements and prove Theorem \ref{intro} by assuming Proposition \ref{hyp'}. In \S\ref{s-l} we prove Proposition \ref{hyp'} for simply laced root systems, and in \S \ref{non-simply-laced} and \S \ref{G2} we deal with the non-simply laced case. In the Appendix, we prove case-by-case several combinatorial properties concerned with simply laced root systems, which are used in the proof of Proposition \ref{hyp'}.

\subsection*{Acknowledgement} We would like to thank Xuhua He for many helpful conversations. Part of the work was done during the second named author's visit to Institute for Advanced Study. He would like to thank the institute for the excellence working atmosphere.

\section{Preliminaries} \label{pre}
\subsection{} \label{setup}
We continue with the notations in the introduction. Moreover, throughout the body of the paper, we assume that $G$ and $T$ split over the valuation ring $\co_F$ of $F$ and that $G$ is simple over $F$.

Let $\car=(Y, \Phi_G^\vee, X, \Phi_G, \SS_0)$ be a based root datum of $G$, where $X$ and $Y$ are the character and cocharacter groups of $T$ respectively together with a perfect pairing $\< , \>: Y \times X \to \ZZ$; $\Phi_G=\Phi \subseteq X$ (resp. $\Phi^\vee \subseteq Y$) is the set of roots (resp. coroots); $\SS_0$ is the set of simple roots appearing in the $\s$-stable Borel subgroup $B \supseteq T$. For $\a \in \Phi$, we denote by $s_\a$ the reflection which sends $\mu \in Y$ to $\mu-\<\mu, \a\> \a^\vee$, where $\a^\vee \in \Phi^\vee$ denotes the corresponding coroot. We say $s_\a$ is a {\it simple reflection} if $\a \in \SS_0$. We also denote by $\SS_0$ the set of simple reflections.

Let $W_0$ be the Weyl group of $T$ in $G$, which is a reflection subgroup of $GL(Y_\RR)$ generated by $\SS_0$. The extended affine Weyl group of $T$ in $G$ is given by $$\tW_G=\tW=N(L) / T(\co_L)=Y \rtimes W_0=\{t^\mu w; \mu \in Y, w \in W_0\}.$$ Let $p: N(L) \to N(L) / T(\co_L) = \tW$ be the quotient map. For $x \in \tW$ we denote by $\dot x \in N(L)$ a lift of $x$ (under $p$). We can embed $\tW$ into the group of affine transformations of $Y_\RR$, where the action of $\tw=t^\mu w$ is given by $v \mapsto \mu+w(v)$. Let $\Phi^+ = \Phi \cap \ZZ_{\ge 0} \SS_0$ be the set of positive roots and let $\ba=\{v \in Y_\RR; 0 < \<\a, v\> < 1, \a \in \Phi^+\}$ be the base alcove.

Let $\tPhi_G=\tPhi=\Phi \times \ZZ$ be the set of (real) affine roots. Let $\tilde \a=(\a, k) \in \tPhi$. We can view $\tilde \a$ as an affine function such that $\tilde \a(v)=-\<\a, v\>+k$ for $v \in Y_\RR$. The induced action of $\tW$ on $\tPhi$ is given by $(\tw(\tilde \a))(v)=\tilde \a(\tw\i(v))$ for $\tw \in \tW$. Let $s_{\tilde \a}=t^{k \a^\vee} s_\a \in \tW$ be the corresponding affine reflection. Then $\{s_{\tilde \a}; \tilde \a \in \tPhi\}$ generates the affine Weyl group $$W^a=\ZZ \Phi^\vee \rtimes W_0=\{t^\mu w; \mu \in \ZZ \Phi^\vee, w \in W_0\}.$$ Moreover, $\tW=W^a \rtimes \Omega$, where $\Omega=\{x \in \tW; x(\ba)=\ba\}$. Set $\tPhi^+=\{\tilde \a \in \tPhi; \tilde \a(\ba) > 0\}$ and $\tPhi^-=-\tPhi^+$. Then $\tPhi=\tPhi^+ \sqcup \tPhi^-$. The associated length function $\ell: \tW \to \NN$ is defined by $\ell(\tw)=|\tPhi^- \cap \tw(\tPhi^+)|$. Let $\SS^a=\{s_{\tilde \a}; \tilde \a \in \tPhi, \ell(s_{\tilde \a})=1\}$. Then $W^a$ is generated by $\SS^a$ and $(W^a, \SS^a)$ is a Coxeter system.

For $\tw, \tw' \in \tW$, we say $\tw \leq \tw'$ if there exist $\tw=\tw_1, \dots, \tw_r=\tw'$ such that $\ell(\tw_k) < \ell(\tw_{k+1})$ and $\tw_k \tw_{k+1}\i$ is an affine reflection for $1 \le k \le r-1$. We call this partial order $\leq$ the Bruhat order on $(\tW, \SS^a)$.

For $\a \in \Phi$, let $U_\a \subseteq G$ be the corresponding root subgroup. We set $$I=T(\co_L) \prod_{\a \in \Phi^+} U_\a(t \co_L) \prod_{\b \in \Phi^+} U_{-\b}(\co_L) \subseteq G(L),$$ which is called a Iwahori subgroup. Here $\co_L$ is the valuation ring of $L$.

\subsection{} \label{admissible} Let $v \in Y_\RR$. We say $v$ is dominant if $\<v, \a\> \ge 0$ for each $\a \in \Phi^+$. We denote by $\bar v$ the unique dominant $W_0$-conjugate of $v$. Let $Y^+$ (resp. $Y_\RR^+$) be the set of dominant vectors in $Y$ (resp. $Y_\RR$). For $\chi_1, \chi_2 \in Y$ we write $\chi_1 \le \chi_2$ if $\chi_2-\chi_1 \in \ZZ_{\ge 0} (\Phi^+)^\vee$, and write $\chi_1 \preceq \chi_2$ if $\bar \chi_1 \le \bar \chi_2$.

For $\tw \in \tW$ there exists a nonzero integer $n$ such that $(\tw)^n=t^\xi$ for some $\xi \in Y$. We define $\nu_{\tw}=\xi / n$, which does not depend on the choice of $n$. For $b', b'' \in G(L)$ we set $\JJ_{b', b''}=\{g \in G(L); g\i b' \s(g)=b''\}$ and put $\JJ_{b'}=\JJ_{b', b''}$ if $b'=b''$.

Let $b \in G(L)$. We denote by $[b]=\{g\i b \s(g); g \in G(L)\}$ the $\s$-conjugate class of $b$. Due to Kottwitz \cite{Ko}, the $\s$-conjugacy class $[b]$ is uniquely determined by two invariants: the Kottwitz point $\eta_G([b]) \in \pi_1(G)$ and the Newton point $\nu_{[b]}^G \in Y_\RR^+$. Here $\eta_G: G(L) \to \pi_1(G)=\tW / W^a \cong Y / \ZZ \Phi^\vee$ is the natural projection which sends $I \tw I$ to $\tw W^a$. To define $\nu_{[b]}^G$, we notice that there exists $x \in \tW$ such that $\dot x \in [b]$. Then $\nu_{[b]}^G=\bar \nu_x$, which does not depend on the choice of $x$.

Let $\l \in Y^+$ and $b \in G(L)$. The {\it closed affine Deligne-Lusztig variety} associated to $(\l, b)$ is defined by $$X(\l, b)=\cup_{\tw \in \Adm(\l)} X_{\tw}(b),$$ where $$\Adm(\l)=\{\tw \in \tW; \tw \leq t^{u(\l)} \text{ for some $u \in W_0$}\}$$ is called the $\l$-admissible set. We refer to \cite{HaNg} and \cite{HH} for more on $\Adm(\l)$. Note that $\Adm(\l') \subseteq \Adm(\l)$ if $\l' \preceq \l$ (see \cite[Lemma 4.5]{Haines}). By \cite{He2}, $X(\l, b) \neq \emptyset$ if and only if $\eta_G(t^\l)=\eta_G([b])$ and $\l-\nu_{[b]}^G \in \RR_{\ge 0} (\Phi^+)^\vee$. Note that $\JJ_b$ acts on $X(\l, b)$ by right multiplication.

We say $\tw \in \tW$ is {\it fundamental} if $I \tw I$ lies in the $I$-$\s$-conjugacy class of $\dot \tw$. By \cite{N1}, $\tw$ is fundamental if and only if it is {\it straight}, that is, $\ell(\tw^n)=n \ell(\tw)$ for each $n \in \ZZ_{\ge 0}$. We refer to \cite{HN} for more descriptions of straight elements.
\begin{prop} [\cite{He}] \label{fund}
Let $x \in \tW$ be a fundamental element. Then $X_x(b) \neq \emptyset$ if and only if $[b]=[\dot x]$, in which case $X_x(b) \cong \JJ_b / (\JJ_b \cap I)$.
\end{prop}

The following result plays a foundational role in this paper.
\begin{prop} [\cite{He2}, \cite{N2}] \label{point}
If $X(\l, b) \neq \emptyset$, then each connected component of $X(\l, b)$ intersects with $X_x(b)$ for some straight element $x \in \Adm(\l)$.
\end{prop}

\subsection{} \label{parabolic} Let $M \subseteq G$ be a Levi subgroup containing $T$. Then $B \cap M \supseteq T$ is a Borel subgroup of $M$. By replacing the triple $(T, B, G)$ with $(T, B \cap M, M)$, we can define, as in \S \ref{setup}, $\Phi_M^+$, $\tW_M$, $\SS_M^a$, $\Omega_M$, $\tPhi_M^+$, $\leq_M$, $I_M$ and so on. For $K \subseteq \SS_M^a$ we denote by $W_K$ the reflection subgroup of $W^a$ generated by $K$. Let $W \supseteq W_K$ be a subgroup of $\tW$. We set $W {}^K=\{x \in W; x \leq x s \text{ for each } s \in K\}$ and ${}^K W=(W {}^K)\i$.

For $v \in Y_\RR$ we set $\Phi_v=\{\a \in \Phi; \a(v)=0\}$ and denote by $M_v$ the Levi subgroup of $G$ generated by $T$ and the root subgroups $U_\a$ for $\a \in \Phi_v$. If $v$ is dominant, we set $J_v=\{s \in \SS_0; s(v)=v\}$. For $J \subseteq \SS_0$ there exists $v \in Y_\RR^+$ such that $J_v=J$. We write $\Phi_J=\Phi_{M_v}$, $\Phi_J^+=\Phi_{M_v}^+$, $\tW_J=\tW_{M_v}$, $\Omega_J=\Omega_{M_v}$,  $\leq_J=\leq_{M_v}$ and so on.
\begin{lem} \label{compare}
Let $M \supseteq T$ be a Levi subgroup of $G$. Let $x, y \in \tW_M$ such that $x \leq_M y$. Then $z' x z\i \leq z' y z\i$ for $z, z' \in \tW^{\SS_M^a}$.
\end{lem}
\begin{proof}
We may assume $y=x s_{\tilde \a}$ for some $\tilde \a \in \tPhi_M^+$. Since $x \leq_M y$, we have $x(\tilde \a) \in \tPhi_M^+$. Let $\tilde \g=z(\tilde \a) \in \tPhi^+$. Then $z' y z\i=z' x z\i s_{\tilde \g}$ and $z' x z\i(\tilde \g)= z' x(\tilde \a)>0$. This means $z' x z\i \leq z' y z\i$ as desired.
\end{proof}

Let $J \subseteq \SS_0$. For $\chi, \chi' \in Y$ we write $\chi \le_J \chi'$ if $\chi'-\chi \in \ZZ_{\ge 0} (\Phi_J^+)^\vee$, and write $\chi \le^J \chi'$ if $\chi \le \chi'+ \phi$ for some $\phi \in \ZZ \Phi_J^\vee$.
\begin{lem} \label{f2}
Let $\tw \in \Omega_J$. Let $\a \in \Phi^+$ and $z \in W_0^J$ such that $s_\a z \in W_0^J$. For $u \in W_J$ we have

(1) $z \tw z\i s_\a \in \Adm(\mu-u\i z\i(\a^\vee)) \cup \Adm(\mu)$;

(2) $s_\a z \tw z\i \in \Adm(\mu+u\i z\i(\a^\vee)) \cup \Adm(\mu)$.
\end{lem}
\begin{proof}
Notice that $t^{z u(\mu)} s_\a \leq t^{z u(\mu)}$ if $\<z u(\mu), \a\> \ge 1$, and $t^{z u(\mu)} s_\a = t^{z u(\mu) - \a^\vee} t^{\a^\vee}s_\a \leq t^{z u(\mu)-\a^\vee}$ if $\<z u(\mu), \a\> \le 0$. By symmetry, $s_\a t^{z u(\mu)} \leq t^{z u(\mu)}$ if $\<z u(\mu), \a\> \le -1$, and $s_\a t^{z u(\mu)} = t^{\a^\vee}s_\a t^{z u(\mu) + \a^\vee} \leq t^{z u(\mu)+\a^\vee}$ if $\<z u(\mu), \a\> \ge 0$. Therefore, $t^{z u(\mu)} s_\a \in \Adm(\mu-u\i z\i(\a^\vee)) \cup \Adm(\mu)$ and $s_\a t^{z u(\mu)} \in \Adm(\mu+u\i z\i(\a^\vee)) \cup \Adm(\mu)$. Notice that $\tw \leq_J t^{u(\mu)}$ and $z, s_\a z \in \tW^{\SS^a_{M_J}}$. By Lemma \ref{compare}, we have $z \tw z\i s_\a \leq t^{z u(\mu)} s_\a \in \Adm(\mu-u\i z\i(\a^\vee)) \cup \Adm(\mu)$ and $s_\a z \tw z\i \leq s_\a t^{z u(\mu)} \in \Adm(\mu+u\i z\i(\a^\vee)) \cup \Adm(\mu)$ as desired.
\end{proof}

\begin{lem} \label{f5}
Assume $\Phi$ is simply laced. Let $\g, \g' \in \Phi^+ - \Phi_J$ such that $\g - \g' \in \ZZ \Phi_J$. Then $\g, \g'$ are conjugate under $W_J$.
\end{lem}
\begin{proof}
We can assume $\g, \g'$ are $J$-dominant. Since $\Phi$ is simply laced, $\g, \g'$ are also $J$-minuscule. So we have $\g=\g'$ since $\g-\g' \in \ZZ \Phi_J$.
\end{proof}

For $D \subseteq \Phi^+$ we denote by $D_{\max}$ (resp. $D_{\max, J}$) the set of maximal elements of $D$ with respect to the partial order $\le$ (resp. $\le_J$).
\begin{lem} \label{f3}
Assume $\Phi$ is simply laced. Let $z \in W_0^J$ and let $D \subseteq \Phi^+ - \Phi_J$ be a $W_J$-stable subset such that $D \cap z\i(\Phi^-) \neq \emptyset$. Then $s_{z(\b)}z=z s_\b \in W_0^J$ for $\b \in (D \cap z\i(\Phi^-))_{\max, J}$.
\end{lem}
\begin{proof}
If $z s_\b \notin W_0^J$, there exists $\g \in \Phi_J^+$ such that $z s_\b (\g) < 0$. Noticing that $\b \neq \pm\g$, we have $\<\b^\vee, \g\>=\<\g^\vee, \b\> \in \{-1, 0, 1\}$ since $\Phi$ is simply laced. If $\<\b^\vee, \g\> \ge 0$, then $z s_\b (\g)=z(\g)-\<\b^\vee, \g\> z(\b) > 0$, a contradiction. So we have $\<\b^\vee, \g\>=-1$ and $s_\g(\b)=\b+\g > \b$, which contradicts our choice of $\b$. The proof is finished.
\end{proof}

\subsection{} \label{chase} Let $\chi, \chi' \in Y$ and $\g \in \SS_0$. Write $\chi \overset \g \to_+ \chi'$ if $\chi'=\chi+\g$ and $\<\chi, \g\> \le -1$. If, moreover, $\<\chi, \g\> = -1$, we write $\chi \overset \g \rightarrowtail_+ \chi'$. We write $\chi \to_+ \chi'$ (resp. $\chi \rightarrowtail_+ \chi'$) if there exist simple roots $\g_1, \cdots, \g_m$ such that $\chi \overset {\g_1} \to_+ \cdots \overset {\g_m} \to_+ \chi'$ (resp. $\chi \overset {\g_1} \rightarrowtail_+ \cdots \overset {\g_m} \rightarrowtail_+ \chi'$ ). On the other hand, write $\chi \overset \g \to_- \chi'$ if $\chi'=\chi-\g$ and $\<\chi, \g\> \ge 1$. We can define $\chi \to_- \chi'$ in a similar way.

We say $\chi \in Y$ is {\it weakly dominant} if $\<\chi, \a\> \ge -1$ for each $\a \in \Phi^+$.
\begin{lem} \label{ind}
Let $\chi, \chi' \in Y$ and $\l \in Y^+$. Then we have

(1) $\chi' \le \l$ if $\chi \le \l$ and $\chi \to_+ \chi'$;

(2) $\chi$ is weakly dominant if $\chi \rightarrowtail_+ \chi'$ and $\chi'$ is dominant;

(3) $\chi \preceq \l$ if $\chi \le \l$ and $\chi$ is weakly dominant;

(4) $\chi'$ is weakly dominant if $\chi \to_- \chi'$ and $\chi$ is weakly dominant.

In particular, $\chi' \preceq \l$ if $\chi \to_+ \chi'$, $\chi' \le \l$ and $\chi'$ is weakly dominant.
\end{lem}
\begin{proof}
(3) is \cite[Proposition 2.2]{Ga}. (2) and (4) follow from \cite[Lemma 6.6]{N2}. (1) is proved in \cite{Ga}, and we repeat its proof here. We can assume $\chi \overset \g \to_+ \chi'$ for some simple root $\g \in \SS_0$. To verify $\chi'=\chi+\g^\vee \le \l$, we prove the coefficient of $\g^\vee$ in $\l-\chi$ is strictly positive. Assume otherwise, then $0 \ge \<\l-\chi, \g\> \ge 0 - (-1) \ge 1$, a contradiction. So (1) is proved.
\end{proof}

For $v \in \RR \Phi$ we have $v=\sum_{\a \in \SS_0} c_\a \a$, where $c_\a \in \RR$ is called the coefficient of $\a$ in $v$. We say $\g \in \Phi^+$ is {\it elementary} if the coefficient of each simple root in $\g$ is at most one.
\begin{lem} \label{elementary}
Suppose there exists a simple root $\a \in \SS_0$ having three neighbors in the Dynkin diagram of $\SS_0$. Let $\g \in \Phi^+$ such that the coefficient of $\a$ in $\g$ is at most one. Then $\g$ is elementary.
\end{lem}

For $\a, \b \in \SS_0$, we denote by $\dist(\a, \b)$ the length of the shortest path (geodesic) connecting $\a$ and $\b$ in the Dynkin diagram of $\SS_0$. For $D, D' \subseteq \SS_0$ we set $\dist(D, D')=\min \{\dist(\g, \g'); \g \in D, \g' \in D'\}$.

\section{Proof of the main result} \label{proof-main}
In this section we outline the proof of Theorem \ref{intro}. We fix $\l \in Y^+$ and $b \in G(L)$ such that $(\l, b)$ is Hodge-Newton irreducible, that is, $\eta_G(t^\l)=\eta_G(b)$ and the coefficient of each simple coroot in $\l-\nu_{[b]}^G$ is strictly positive. In particular, $X(\l, b) \neq \emptyset$.

Following \cite{V1}, we say $x \in \tW$ is {\it short} if $\nu_x$ is dominant and $x \in \Omega_{J_{\nu_x}}$.
\begin{lem} [\cite{V1}] \label{dom}
There exists a unique short element $\tw$ such that $\dot \tw \in [b]$. Moreover, $\tw \in t^\mu W_0$ with $\mu \preceq \l$ weakly dominant.
\end{lem}
\begin{proof}
By Proposition \ref{point}, there exists a straight element $x \in \Adm(\l)$ such that $[b]=[\dot x]$. In particular, $\bar \nu_x=\nu_{[b]}^G$. Let $z \in W_0^{J_{\bar \nu_x}}$ such that $\nu_x=z(\nu_{[b]}^G)$. We set $\tw =z\i x z$ and hence $\dot \tw \in [\dot x]=[b]$. By \cite[Lemma 4.1]{N2}, $\tw \in \Omega_{J_{\nu_{\tw}}}$ and $\mu$ is weakly dominant. Since $x \in \Adm(\l)$ and $W_0 x W_0 = W_0 \tw W_0$, we have $\mu \preceq \l$ as desired. The uniqueness of $\tw$ follows from \cite[Lemma 5.3]{V1}.
\end{proof}

Let $\tw$ be the short element with $\dot \tw \in [b]$. Let $\mu \in Y$ such that $\tw \in t^\mu W_0$. Then $\mu$ is $J_{\nu_{\tw}}$-dominant and $J_{\nu_{\tw}}$-minuscule. Let $J \subseteq J_{\nu_{\tw}}$ be the union of connected components $H$ of $J_{\nu_{\tw}}$ such that $\mu$ is noncentral on $\Phi_H$. Let $K=\{s \in J; s(\mu)=\mu\}$. Then $x=t^\mu w_K w_J$, where $w_K$ and $w_J$ are the unique longest elements of $W_K$ and $W_J$ respectively.
\begin{lem}\label{add-simple}
We have $\mu+\a^\vee \le \l$ for $\a \in \SS_0 - J$.
\end{lem}
\begin{proof}
By definition, the coefficient of $\a$ in $\l-\nu_{[b]}^G$ is strictly positive. On the other hand, $\nu_{[b]}^G=\nu_{\tw} \in \mu + \RR \Phi_J^\vee$ since $\tw \in \tW_J$. Thus the coefficient of $\a$ in $\l-\mu \in \ZZ \Phi^\vee$ is also strictly positive as desired.
\end{proof}

\begin{lem} \label{positive}
If $\g \in \Phi^+ - \Phi_{J_{\nu_{\tw}}}$ is $J$-dominant, then $\<\mu, \g\> \ge 1$.
\end{lem}
\begin{proof}
By definition, $\mu-\nu_{\tw}=\mu-\nu_{[b]}^G \in \RR \Phi_J^\vee$ is $J$-dominant. So $\mu-\nu_{\tw} \in \RR_{\ge 0} (\Phi_J^+)^\vee$. Since $\g$ is $J$-dominant, we have $\<\mu, \g\>=\<\nu_{\tw}, \g\>+\<\mu-\nu_{\tw}, \g\> \ge \<\nu_{\tw}, \g\> > 0$ as desired.
\end{proof}

For $\chi \in Y$ we denote by $\chi_J$ (resp. $\chi^J$) the unique $J$-antidominant (resp. $J$-dominant) $W_J$-conjugate of $\chi$.
\begin{cor} \label{f4}
Let $z \in W_0^J$ and $\g \in \Phi^+ - \Phi_J$ such that $z \geq z s_\g \in W_0^J$ (resp. $z \leq z s_\g \in W_0^J$). Then $z \tw s_\g z\i, z s_\g \tw z\i \in \Adm(\l)$ if $z \tw s_\g z\i \in \Adm(\l)$ (resp. $z s_\g \tw z\i \in \Adm(\l)$). Moreover, the latter condition holds if $\mu+\g_J^\vee \preceq \l$.
\end{cor}
\begin{proof}
By symmetry, we can assume $z \geq z s_\g$. If $\g \in \Phi_{J_{\nu_{\tw}}}^+ - \Phi_J$, we have $z s_\g \tw z\i = z \tw s_\g z\i \in \Adm(\l)$. Otherwise, by Lemma \ref{positive} we have $\<\mu, \g^J\> \ge 1$ and hence $\mu-(\g^J)^\vee \preceq \mu \preceq \l$. So $z s_\g \tw z\i \in \Adm(\l)$ by Lemma \ref{f2}. The ``Moreover" part follows from Lemma \ref{f2}.
\end{proof}

\begin{lem} \label{shrink}
Let $\g \in \Phi^+ - \Phi_J$ such that $\<\mu, \g_J\>=\<\mu, \g\>=0$. Then $\g \le_J w\i(\g)$.
\end{lem}
\begin{proof}
By assumption $\g_J \le_K \g$ and hence $\g=u(\g_J)$ for some $u \in W_K$. Since $w_J w_K \in W_J^K$ and $\ell((w_J w_K) (w_K u))=\ell(w_J w_K) + \ell(w_K u)$, we have $w_K u \leq  (w_J w_K) (w_K u)=w_J u$. Moreover, since $\g_J$ is $J$-antidominant, we have $w_K(\g)=w_K u (\g_J) \le_J w_J u (\g_J)=w_J(\g)$, that is, $w\i(\g) \ge_J \g$ as desired.
\end{proof}

Following \cite[Section 8]{N2}, we define a subset $C_{\l, J, b}$ of $\Phi^+$ as follows. If $G$ is of type $G_2$ and $J=\{s_\g\}$ with $\g$ the unique short simple root, we define $$C_{\l, J, b}=\{\text{the unique long simple root}\}.$$ Otherwise, we define $$C_{\l, J, b}=\{\a \in \Phi^+ - \Phi_J; \mu+\a^\vee \preceq \l, \text{ $\a^\vee$ is $J$-minuscule and $J$-anti-dominant}\}.$$

\begin{lem} \label{span} \cite[Lemma 8.1 \& Corollary 8.2]{N2}
The quotient $\ZZ \Phi^\vee / \ZZ \Phi_J^\vee$ is spanned by $\a^\vee$ for $\a \in C_{\l, J, b}$.
\end{lem}

\begin{lem} \label{teq}
Let $\a \in C_{\l, J, b}$ and $s=t^{w_J(\a^\vee)} s_{w_J(\a)}$. Then we have

(a) $\tw, s \tw, \tw s, s \tw s \in \Adm(\l)$;

(b) $(\ZZ \a + \ZZ w(\a)) \cap \Phi$ is simply laced;

(c) $\<\mu, w_J(\a)\> \ge 2$ if $w_J(\a)+ w_K(\a) \in \Phi$.
\end{lem}
\begin{proof}
Since $\tw \in \Omega_J$, we have $\tw \leq_J t^\mu \in \Adm(\l)$. By definition, $w_J(\a^\vee)$ is $J$-dominant and $J$-minuscule. Thus $s \in \tW^{\SS_{M_J}^a}$. Thanks to Lemma \ref{compare}, $s \tw s \in \Adm(\l)$. Moreover, we have $$\tw s \leq t^{w_J(\mu)} s =t^{w_J(\mu)+w_J(\a^\vee)} s_{w_J(\a)} \leq t^{w_J(\mu+\a^\vee)} \in \Adm(\l),$$ where the first inequality follows form Lemma \ref{compare} (since $\tw \leq_J t^{w_J(\mu)}$); the second inequality follows from the fact $\<w_J(\mu + \a^\vee), w_J(\a)\>=\<\mu, \a\>+2 \ge -1+2=1$ (since $\mu$ is weakly dominant). If $\a \in \Phi^+ - \Phi_{J_{\nu_{\tw}}}$, then $\<\mu, w_J(\a)\> \ge 1$ (by Lemma \ref{positive}) and $w\i (w_J(\a)) > 0$, which implies $s \tw \leq \tw  \in \Adm(\l)$ as desired. If $\a \in \Phi_{J_{\nu_{\tw}}}^+ - \Phi_J$, we have $s \tw=\tw s \in \Adm(\l)$. Thus (a) is proved.

(b) is proved in the proof of \cite[Propostiion 8.3]{N2}.

To prove (c), we assume $w_J(\a) + w_K(\a) \in \Phi$  and $\<\mu, w_J(\a)\> \le 1$, and show this will lead to a contradiction. Since $w_J(\a^\vee)$ is $J$-dominant and $J$-minuscule, by \cite[Lemma 4.6.1]{CKV}, there exists $\g_i \in \Phi_J^+$ with $1 \le i \le m$ such that $w_J(\a^\vee) - w w_J(\a^\vee)= w_J(\a^\vee)-w_K(\a^\vee)=\sum_{1 \le i \le m} \g_i^\vee$, $\<\g_i^\vee, \g_j\>=0$ and $\<w_J(\a^\vee), \g_i \>=\<\mu, \g_i\>=1$ for any $1 \le i \neq j \le m$. In particular, $w_K(\a^\vee)=s_{\g_1} \cdots s_{\g_m} (w_J(\a^\vee))$. Since $w_J(\a) + w_K(\a) \in \Phi$ and $(\ZZ \a + \ZZ w(\a)) \cap \Phi$ is simply laced, we have $$ -1=\<w_K(\a^\vee), w_J(\a)\>= 2 - \sum_{1 \le i \le m} \<\g_i^\vee, w_J(\a)\>,$$ that is, $\sum_{1 \le i \le m} \<\g_i^\vee, w_J(\a)\>=3$. On the other hand, by $\<\mu, w_J(\a)\> \le 1$ we have \begin{align*}\<\mu, w_K(\a)\> &=\<\mu, s_{\g_1} \cdots s_{\g_m} (w_J(\a))\> \\ &=\<\mu, w_J(\a)\>-\sum_{1 \le i \le m} \<\g_i^\vee, w_J(\a)\> \<\mu, \g_i\> \\ &\le 1 - 3 =-2, \end{align*} which contradicts that $\mu$ is weakly dominant. So (c) is proved.
\end{proof}

For $b', b'' \in G(L)$ we set $\JJ_{b', b''}=\{g \in G(L); g\i b' \s(g)=b''\}$. For $P, P' \in X(\l, b)$ we write $P \sim_{\l, b} P'$ if they are connected in $X(\l, b)$.
\begin{prop} \label{generator}
Let $\a \in C_{\l, J, b}$ and $s=t^{w_J(\a^\vee)} s_{w_J(\a)}$. Then we have $g_0 I \sim_{\l, b} g_0 \dot s I$ for $g_0 \in \JJ_{b, \dot \tw}$.
\end{prop}
\begin{proof}
Set $\b=w_J(\a)$. Define $\textsl{g}: \AA^1 \to X(\l, b)$ by $\textsl{g}(z)=g_0 U_{-\b}(z t\i) I$. We extend $\textsl{g}$ to a (unique) morphism from $\PP^1=\AA^1 \sqcup \{\infty\}$ to $X(\l, b)$, which is still denoted by $\textsl{g}$. Then $\textsl{g}(0)=g_0 I$ and $\textsl{g}(\infty)=g_0 \dot s I$. It remains to show the image of $\textsl{g}$ lies in $X(\l, b)$. Let $c_0 \in \co_L^\times$ be the constant such that $\dot \tw\i U_{-\b}(z t\i) \dot \tw=U_{-w\i(\b)}(c_0 z t^{\<\mu, \b\>-1})$.

If $\a \in \Phi^+ - \Phi_{J_{\nu_{\tw}}}$, then $\b \neq -w\i(\b)$ and $\<\mu, \b\> \ge 1$ by Lemma \ref{positive}. Moreover, thanks to Lemma \ref{teq} (b), the root system $\Phi \cap (\ZZ \b + \ZZ w\i(\b))$ is simply laced and $\<\mu, \b\> \ge 2$ if $\b+w\i(\b) \in \Phi$. Thus $$U_{-\b}(-z^q t\i) U_{-w\i(\b)}(-c_0 z t^{\<\mu, \b\>-1}) U_{-\b}(z^q t\i) \in I.$$ Therefore, \begin{align*} \textsl{g}(z)\i b \s(\textsl{g}(z)) &= U_{-\b}(-z t\i) g_0\i b \s(g_0) U_{-\b}(z^q t\i) \\ &=U_{-\b}(-z t\i) \dot \tw U_{-\b}(z^q t\i)\\ &=\dot \tw U_{-\b}(z^q t\i) U_{-\b}(-z^q t\i) U_{-w\i(\b)}(-c_0 z t^{\<\mu, \b\>-1}) U_{-\b}(z^q t\i) \\ & \in \dot \tw U_{-\b}(z^q t\i) I \\ &\subseteq I \tw I \cup I \tw s I \subseteq  I \Adm(\l) I, \end{align*} where the last inclusion follows from Lemma \ref{teq} (a).

If $\a \in \Phi_{J_{\nu_{\tw}}}^+ - \Phi_J$, we have $\b=\a$ and $\tw$ commutes with $s_\b$. Thus $$\textsl{g}(z)\i b \s(\textsl{g}(z))=\dot \tw U_{-\b} ((z^q-c_0 z) t\i) \in I \tw I \cup I \tw s I \subseteq I \Adm(\l) I.$$

So the image of $\textsl{g}$ always lies in $X(\l, b)$ as desired.
\end{proof}

\

Let $\g \in \Phi^+$ and $x \in \tW$. We say $\g$ is $x$-permissible if $$U_\g(y_1) x U_\g(y_2) \subseteq I \{x, x s_\g, s_\g x, s_\g x s_\g\} I$$ for any $y_1, y_2 \in \co_L$.

Let $z, z' \in W_0^J$ and $\g \in \Phi^+$. Write $z \overset \g \leftrightarrow z'$ with $z'=s_\g z$ if $z \tw z\i s_\g, s_\g z \tw z\i \in \Adm(\l)$ and $\g$ is $z \tw z\i$-permissible, or equivalently by Lemma \ref{bound}, $z' \tw {z'}\i$-permissible. We write $z \leftrightarrow z'$ if there exist positive roots $\g_1, \dots, \g_n$ such that $z \overset {\g_1} \leftrightarrow \cdots \overset {\g_n} \leftrightarrow z'$. By Lemma \ref{compare}, we always have $z \tw z\i, z' \tw {z'}\i \in \Adm(\l)$.
\begin{prop} \label{hyp}
Let $z, z' \in W_0^J$ and $g_z \in \JJ_{b, \dot z \dot \tw \dot z\i}$. If $z \overset \g \leftrightarrow z'$ for some $\g \in \Phi^+$, then $g_z I \sim_{\l, b} g_z \dot s_\g I$. As a consequence, if $z \leftrightarrow z'$, then $g_z I \sim_{\l, b} g_z \dot z (\dot{z}')\i I$.
\end{prop}
\begin{proof}
Define $\textsl{g}: \PP^1 \to X(\l, b)$ by $\textsl{g}(y)=g_z U_{\g}(y) I$. Then $\textsl{g}(0)=g_z I$ and $\textsl{g}(\infty)=g_z \dot s_\g I$. Since $\g$ is $z \tw z\i$-permissible, the image of $\textsl{g}$ lies in $X(\l, b)$ and hence $g_z I, g_z \dot s_\g I$ are connected in $X(\l, b)$.
\end{proof}

\begin{prop} \label{hyp'}
Let $z \in W_0^J$. If $z \neq 1$, then there exists $\g \in \Phi^+$ such that $z \geq z s_\g \in W_0^J$ and $z \leftrightarrow z s_\g$. As a consequence, we have $z \leftrightarrow z'$ for any $z' \in W_0^J$.
\end{prop}

\begin{cor} \label{transitive}
Each connected component of $X(\l, b)$ intersects with $\JJ_{b, \dot \tw} I / I$. In particular, $\JJ_b$ acts transitively on $\pi_0(X(\l, b))$.
\end{cor}
\begin{proof}
Let $\cc$ be a connected component of $X(\l, b)$. By Proposition \ref{point}, $\cc$ intersects with $X_x(b)$ for some straight element $x \in \Adm(\l)$. Therefore, by Proposition \ref{fund}, there exists $g \in \JJ_{b, \dot x}$ such that $g I \in \cc$. Let $z \in W_0^{J_{\nu_{\tw}}} \subseteq W_0^J$ such that $z\i(\nu_x)=\nu_{[b]}^G=\nu_{\tw}$. By Lemma \ref{dom}, $x=z \tw z\i$. Applying Proposition \ref{hyp} (where we take $z'=1$), we have $g I \sim_{\l, b} g \dot z I$ and hence $g \dot z I \in (\JJ_{b, \dot \tw} I / I) \cap \cc$ as desired. Since $\JJ_b$ acts transitively on $\JJ_{b, \dot \tw}$ by left multiplication, $\JJ_b$ acts transitively on $\pi_0(X(\l, b))$.
\end{proof}

Now we are ready to prove  Theorem \ref{intro}.

\

Let $\JJ_{\dot \tw}^\circ=\ker \eta_G \cap \JJ_{\dot \tw}$ and $\JJ_b^\circ=\ker \eta_G \cap \JJ_b=g_0 \JJ_{\dot \tw}^\circ g_0\i$ with $g_0 \in \JJ_{b, \dot \tw}$. By Corollary \ref{transitive}, $\JJ_b$ acts transitively on $\pi_0(X(\l, b))$). Since $\JJ_b^\circ$ is a normal subgroup of $\JJ_b$, it suffices to show $\JJ_b^\circ$ fixes the connected component $\cc$ of $X(\l, b)$ which contains $g_0 I$. Thanks to \cite[Theorem 5.5]{HZ}, $\JJ_{\dot \tw}^\circ$ is generated by $\JJ_{\dot \tw} \cap I_{M_{J_{\nu_{\tw}}}}$, $\JJ_{\dot \tw} \cap p\i(W^a_{J_{\nu_{\tw}}})$ and $\JJ_{\dot \tw} \cap p\i(\Omega_{J_{\nu_{\tw}}}^\circ)$, where $\Omega_{J_{\nu_{\tw}}}^\circ=\ker \eta_G \cap \Omega_{J_{\nu_{\tw}}} \cong \ZZ\Phi^\vee / \ZZ\Phi_{J_{\nu_{\tw}}}^\vee$. Notice that $$\JJ_{\dot \tw} \cap p\i(W^a_{J_{\nu_{\tw}}})= (\JJ_{\dot \tw} \cap p\i(W_J^a)) (\JJ_{\dot \tw} \cap p\i(W_{J_{\nu_{\tw}} - J}^a)).$$ Moreover, $\JJ_{\dot \tw} \cap p\i(\Omega_{J_{\nu_{\tw}}}^\circ)$ lies in the group generated by $\JJ_{\dot \tw} \cap p\i(W_{J_{\nu_{\tw}} - J})$ and $\JJ_{\dot \tw} \cap p\i(\Omega_J^\circ)$, where $\Omega_J^\circ=\ker \eta_G \cap \Omega_J$.

Firstly, we show $g_0 (\JJ_{\dot \tw} \cap I) g_0\i$ fixes $\cc$. This follows by observing that $g_0 (\JJ_{\dot \tw} \cap I) g_0\i$ fixes $g_0 I \in \cc$.

Secondly, we show $g_0 (\JJ_{\dot \tw} \cap p\i(\Omega_J^\circ)) g_0\i$ fixes $\cc$. Let $\g \in C_{\l, J, b}$. By definition, $\g$ is $J$-antidominant and $J$-minuscule. So there exits a unique element, denoted by $y_{\g}$, lying in $\Omega_J \cap t^{w_J(\g^\vee)} W_J \subseteq \Omega_J^\circ$. Since $\tw y_\g =y_\g \tw$ (because $\Omega_J$ is commutative), there exists a lift $\dot y_\g \in N(L)$ of $y_\g$ which also lies in $\JJ_{\dot \tw}$. By Lemma \ref{span}, $\Omega_J^\circ$ is spanned by $C_{\l, J, b}$. So it suffices to show each $g_0 \dot y_\g g_0\i$ fixes $\cc$, that is, $g_0 I \sim_{\l, b} g_0 \dot y_\g I$. Set $\b=w_J(\g)$ and $s=t^{\b^\vee} s_\b$. Then $s=y_\g z\i$ for some $z \in W_0$. Noticing that $y_\g \in \Omega_J$ and $s=z y_\g\i \in \tW^{\SS_{M_J}^a}$, we have $z \in W_0^J$. Since $g_0 \dot y_\g \dot z\i \in \JJ_{b, \dot z \dot \tw \dot z\i}$, by Proposition \ref{hyp} we have $g_0 \dot s I=g_0 \dot y_\g \dot z\i I \sim_{\l, b} g_0 \dot y_\g I$. On the other hand, by Proposition \ref{generator} we have $g_0 I \sim_{\l, b} g_0 \dot s I$. So $g_0 I \sim_{\l, b} g_0 \dot y_\g I$ as desired.

Thirdly, we show $g_0 (\JJ_{\dot \tw} \cap p\i(W_J^a)) g_0\i$ fixes $\cc$. Since $\dot \tw$ is basic in $M_J(L)$ and $\mu$ is noncentral on each connected component of $J$, by \cite[Theorem 6.3]{HZ} $\JJ_{\dot \tw} \cap p\i(W_J^a)$ acts trivially on $\pi_0(X^{M_J}(\mu, \dot \tw))$, which, coupled with the natural inclusion $X^{M_J}(\mu, \dot \tw) \subseteq X(\l, \dot \tw)$, implies $g_0 (\JJ_{\dot \tw} \cap p\i(W_J^a)) g_0\i$ fixes $\cc$.

Finally, we show  $g_0 (\JJ_{\dot \tw} \cap p\i(W_{J_{\nu_{\tw}}-J}^a) g_0\i$ fixes $\cc$. Notice that $W_{J_{\nu_{\tw}}-J}^a$ is generated by $s_\a$ and $t^{\a^\vee}$ for $\a \in \Phi_{J_{\nu_{\tw}}-J}^+$. Let $\dot s_\a \in N(L) \cap \JJ_{\dot \tw}$ be a lift of $s_\a$. As we have already shown $g_0 t^{\a^\vee} g_0\i \in g_0 (\JJ_{\dot \tw} \cap p\i(\Omega_J^\circ)) g_0\i$ fixes $\cc$, it suffices to prove $g_0I \sim_{\l, b} g_0 \dot s_\a I$. This follows from Proposition \ref{hyp} by noticing that $s_\a \in W_0^J$. The proof is finished.

\

In the remainder of this section, we deduce a criterion of $z \tw z\i$-permissible elements for $z \in W_0^J$.
\begin{lem} \label{bound}
Let $z \in W_0^J$, $\tw'=z \tw z\i$ and $\a \in \Phi^+ - z(\Phi_J)$. Then $\a$ is $z \tw z\i$-permissible if and only if one of the following conditions fails:

(1) $\<\mu', \a\>=\<\mu', w'(\a)\>=0$;

(2) $w'(\a), {w'}\i(\a) < 0$;

(3) $\a+w'(\a), \a+{w'}\i(\a) \in \Phi^+$;

(4) $\<\a^\vee, w'(\a)\> = -1$.

Here $\mu' \in Y$ and $w' \in W_0$ such that $\tw'=t^{\mu'} w'$.

As a consequence, if, moreover, $s_\a z \in W_0^J$, then $\a$ is $\tw'$-permissible if and only if $\a$ is $s_\a \tw' s_\a$-permissible.
\end{lem}
\begin{proof}
Note that $\mu'=z(\mu)$ and $w'=z w z\i=z w_K w_J z\i$. We have

(a) $U_\a(y_1) \tw' \in I \tw' I$ if $s_\a \tw' < \tw'$ and $U_\a(y_1) \tw' \in I s_\a \tw' I$ otherwise for any $y_1 \in \co_L^\times$.

(b) If $\a \neq w'(\a)$ and $\<\mu', \a-w'(\a)\>=0$, then $|\SS_0| \ge 3$ and hence $\Phi \cap (\ZZ \a + \ZZ w'(\a))$ is not of type $G_2$.

Now we prove (b). Assume otherwise, that is, $|\SS_0| \le 2$. Since $w'(\a) \neq \a$, we have $J \neq \emptyset$ and hence $|J|=1$ (since $J \neq \SS_0$). In particular, $w'=s_\b$ for some $\b \in z (\Phi_J^+)$ with $\<\mu', \b\>=1$ (since $\tw \in \Omega_J$). So $\a-w'(\a) \in \ZZ \b -\{0\}$ and $\<\mu', \a-w'(\a)\> \neq 0$, a contradiction. Therefore (b) is proved.

($\Leftarrow$) We assume $\a$ is not $\tw'$-permissible and show (1), (2), (3) and (4) are satisfied.

Suppose $\<\mu', \a\> \neq 0$. If $\<\mu', \a\> \le -1$, then for $y_1, y_2 \in \co_L^\times$ we have $$U_\a(y_1) \tw' U_\a(y_2) \in \tw' U_\a(y_2) I \subseteq I \tw' I \cup I \tw' s_\a I.$$ If $\<\mu', \a\> \ge 1$, then \begin{align*} U_\a(y_1) \tw' U_\a(y_2) & \in  I s_\a U_{-\a} (y_1\i) t^{\mu'} w' U_\a(y_2) \\ & \subseteq  I s_\a \tw' U_{-{w'}\i(\a)}(t^{\<\mu', \a\>} y_1\i) U_\a(y_2) \\ & \subseteq I s \tw' U_\a(y_2) I \\ &\subseteq I s_\a \tw' I \cup I s_\a \tw' s_\a I, \end{align*} which contradicts our assumption. So $\<\mu', \a\> =0$. Similar argument also shows that $\<\mu', w'(\a)\>=0$. Therefore, (1) is verified.

Since $\a$ is not $\tw'$-permissible, by (1) and (a) we have $w'(\a) \neq \a$. Applying (b) we have

(b') $\Phi \cap (\ZZ \a + \ZZ w'(\a))$ is not of type $G_2$.

We claim that

(c) if $\a-w'(\a) \in \Phi^+$ (resp. $\a-{w'}\i(\a) \in \Phi^+$), then ${w'}\i(\a-w'(\a)) > 0$ (resp. $w'(\a-{w'}\i(\a)) > 0$).

Indeed, since $\a-w'(\a) \in z(\ZZ \Phi_J)$ and $\<\mu', \a-w'(\a)\>=0$, we deduce that $\a-w'(\a) \in z(\Phi_J^+)$ and hence $\a-w'(\a) \in z(\Phi_K^+)$. Then $\a-{w'}\i(\a) \in \Phi^+$ follows from the observation that ${w'}\i z(\Phi_K^+) \subseteq \Phi^+$. The other statement follows in the same way by replacing the triple $(\mu', w', K)$ with $(-w^{\prime-1}(\mu'), {w'}\i, w\i(K))$. So (c) is proved.

Moreover, we have

(d) If $w'(\a) > 0$ (resp. ${w'}\i(\a) > 0$), then $\a-w'(\a) \in \Phi^+$ (resp. $\a-{w'}\i(\a) \in \Phi^+$).

Without loss of generality, we assume $w'(\a) > 0$. By (1), (b') and the fact that $\a \neq \pm w'(\a)$ are of the same length, one computes that \begin{align*} U_\a(y_1) \dot \tw' U_\a(y_2) &\in  U_\a(y_1) U_{-w'(\a)} (c y_2\i) \dot \tw' \dot s_\a I \\ &\subseteq \begin{cases} I U_{\a-w'(\a)}(c y_1 y_2\i) U_\a(y_1) \dot \tw' \dot s_\a I, & \text{ if } \a-w'(\a) \in \Phi; \\ I U_\a(y_2) \dot \tw \dot s_\a I, & \text{ otherwise.} \end{cases} \end{align*} for some constant $c \in \co_L^\times$. So $\a-w'(\a) \in \Phi^+$ (since $\a$ is not $\tw'$-permissible) and (d) is proved.

Suppose $w'(\a) > 0$. By (d) and (c), ${w'}\i(\a-w'(\a)) > 0$ and hence ${w'}\i(\a)=\a+{w'}\i(\a-w'(\a)) > 0$. Applying (d) again, we have $\a-{w'}\i(\a) > 0$, a contradiction. So $w'(\a) < 0$. Similarly, ${w'}\i(\a)<0$ and (2) is verified.

By (1) and (2), we have \begin{align*} U_\a(y_1) \dot \tw' U_\a(y_2) &= U_\a(y_1) U_{w'(\a)}(c' y_2) \dot \tw' \\ &\in \begin{cases} I U_{\a+w'(\a)}(c' y_1 y_2) U_\a(y_1) \dot \tw' I, &\text{ if } \a+w'(\a) \in \Phi; \\ I U_\a(y_1) \dot \tw' I, &\text{ otherwise } \end{cases} \end{align*} for some constant $c' \in \co_L^\times$. Thus $\a+w'(\a) \in \Phi^+$. Similarly, we have $\a+{w'}\i(\a) \in \Phi^+$ and (3) is verified.

Suppose $\<\a^\vee, w'(\a)\> \ge 0$. Then, by (b'), $\Phi \cap (\ZZ \a + \ZZ w'(\a))$ is of type $B_2$ since $\a$ and $w'(\a)$ are of the same length and $\a+w'(\a) \in \Phi^+$. In particular, $\<\a^\vee, w'(\a)\> =0$ and $\a-w'(\a) \in \Phi^+$. Thus, by (c) we have ${w'}\i(\a)=\a+{w'}\i(\a-w'(\a)) > 0$, which contradicts (2). So $\<\a^\vee, w'(\a)\> \le -1$ and hence $\<\a^\vee, w'(\a)\>=-1$ (since $\a \neq \pm w'(\a)$ are of the same length). Therefore, (4) is verified.

($\Rightarrow$) Assume (1), (2), (3) and (4) hold. One deduces that $\a+w'(\a) \in \Phi^+$, $\tw' < s_{\a+w'(\a)} \tw'$ and $\Phi \cap (\ZZ(\a) + \ZZ(w'(\a)))$ is of type $A_2$. Thus, for $y_1, y_2 \in \co_L^\times$ we have $$U_\a(y_1) \dot \tw' U_\a(y_2) \in I U_{\a+w'(\a)}(c' y_1 y_2) \dot \tw' I \subseteq I \dot s_{\a+w'(\a)} \dot \tw' I.$$ Since $s_{\a+w'(\a)} \tw' \notin \{\tw' s_\a, s_\a \tw'\}$, $\a$ is not $\tw'$-permissible as desired.
\end{proof}

\section{Simply laced Dynkin diagrams} \label{s-l}
In this section we prove Proposition \ref{hyp'} for simply laced root system $\Phi$. Let $\l, \mu, J, \tw, w$ be as in Section \ref{proof-main}. Fix $1 \neq z \in W_0^J$.

\begin{lem}\label{seq}
Let $\a, \b \in \SS_0 - J$ such that

(i) $\<\mu, \a\> \ge 0$ and $\<\mu, \b\>=-1$;

(ii) $\<\mu, \d_i\> \ge 0$ for $2 \le i \le m$, where $\b=\d_1, \d_2, \dots, \d_m=\a$ is a geodesic (shortest path) in the Dynkin diagram of $\SS_0$.

If $\mu+\a^\vee$ is not weakly dominant, then $\sum_{k=2}^{m-1} \<\mu, \d_k\> \le 1$ and one of the following cases occurs:

(1) $\mu + \d_m^\vee + \dots + \d_1^\vee$ is weakly dominant;

(2) $\SS_0$ is of type $E_8$, $\mu=k \o_1^\vee+\o_3^\vee-\o_4^\vee+\o_6^\vee$ (resp. $\mu=\o_1^\vee-\o_4^\vee+\o_5^\vee+k\o_8^\vee$) with $k \in \ZZ_{\ge 0}$, $\a=\a_1$ (resp. $\a=\a_8$), $\b=\a_4$, $J_{\nu_{\tw}}=\SS_0-\{\a, \b\}$ and $J =J_{\nu_{\tw}}-\{\a_2\}$. In this case, $\mu + \d_m^\vee + \dots + \d_1^\vee +\xi_1^\vee + \xi_2^\vee + \b^\vee$ is weakly dominant, where $\{\xi_1, \xi_2\}=\{\a_2, \a_5\}$ (resp. $\{\xi_1, \xi_2\}=\{\a_2, \a_3\}$).

(3) $\SS_0$ is of type $E_8$, $\mu=\o_1^\vee - \o_5^\vee + \o_6^\vee + k \o_8^\vee$ with $k \in \ZZ_{\ge 0}$, $\a=\a_8$, $\b=\a_5$ and $J_{\nu_{\tw}}=J=\SS_0-\{\a_5, \a_8\}$. In this case, $\mu + \d_m^\vee + \dots + \d_1^\vee + \e +\xi_1^\vee + \xi_2^\vee + \e^\vee + \b^\vee$ is weakly dominant, where $\e=\a_4$ and $\{\xi_1, \xi_2\}=\{\a_2, \a_3\}$.

Here, in (2) and (3), we use the labeling of the type $E_8$ Dynkin diagram as in \cite{Hum} by the integers $1 \le i \le 8$, and $\a_i$ and $\o_i^\vee$ denote the corresponding simple root and fundamental coweight respectively.
\end{lem}
The proof is given in \S \ref{prf-seq}.

\

Now we define \begin{align*}\Xi_1^+(\mu)&=\{\a \in \Phi^+ - \Phi_J; \<\mu, \a_J\> =-1\} \\ \Xi^+(\l, \mu)&=\{\a \in \Phi^+; \mu+\a_J^\vee \preceq \l\} \supseteq \Xi_1^+(\mu).\end{align*}
\begin{lem}\label{empty}
Assume $\Phi$ is simply laced. Then for each nonempty subset $D \subseteq \Xi_1^+(\mu)$ there exists $\a \in D_{\max, J}$ such that one of the following statements fails:

(1') $\<\mu, \a\>=\<\mu, w(\a)\>=0$;

(2') $w(\a), w\i(\a) \notin D$;

(3') $\a+w\i(\a) \in \Phi - \Xi_1^+(\mu)$;

(4') $w\i(\a)-\e \in D$ if $w\i(\a)-\e, \a+\e \in \Phi$ for some $\e \in \Phi_J^+$;

(5') $\a' \in D$ if $\a' \in \Phi^+$ and $\a' \le_J \a$.
\end{lem}
The proof is given in \S \ref{prf-empty}.

\subsection{}\label{hyp1} First we show Proposition \ref{hyp'} holds if $$D':=\Xi^+(\l, \mu) \cap z\i(\Phi^-) \neq \emptyset.$$ We claim that there exist $\d \in D'_{\max, J}$ such that $z \geq z s_\d \in W_0^J$ and $z \overset {-z(\d)} \longrightarrow z s_{\d}$. Let $\g \in D'_{\max, J}$. Then $z \leq z s_\g \in W_0^J$ (by Lemma \ref{f3}) and $\mu+\g_J^\vee \preceq \l$. By Corollary \ref{f4} we have $z \tw s_\g z\i, z s_\g \tw z\i \in \Adm(\l)$. Thus, to prove the claim, it suffices to show in $-z(D'_{\max, J}) \subseteq \Phi^+$ there exists a $z \tw z\i$-permissible root. Assume otherwise, that is, the four statements of Lemma \ref{bound} holds for each root in $-z(D'_{\max, J})$. We show this will lead to a contradiction.

First we show $D' \subseteq \Xi_1^+(\mu)$. Let $\g \in D'_{\max, J}$. By Lemma \ref{bound} (1) we have $0=\<\mu, w(\g)\>=\<\mu, \g\> \ge \<\mu, \g_J\>$. If $\g \notin \Xi_1^+(\mu)$, then $\<\mu, \g_J\> \ge 0$ (since $\mu$ is weakly dominant) and hence $\<\mu, \g_J\>=0$. By Lemma \ref{shrink}, $w(\g) \le_J w\i(w(\g))=\g$ and hence $z(w(\g)) \le z(\g) < 0$, which contradicts to Lemma \ref{bound} (2). So $\g \in \Xi_1^+(\mu)$ and hence $D' \subseteq \Xi_1^+(\mu)$.

Fix a maximal element $\g_0 \in D'$ with respect to $\le^J$ (see \S\ref{parabolic}). Define $$\emptyset \neq D=\{\g \in D'; \g-\g_0 \in \ZZ \Phi_J\} \subseteq \Xi_1^+(\mu).$$ We show each $\a \in D_{\max, J}$ satisfies the five conditions of Lemma \ref{empty}, which is a contradiction. Indeed, one checks that (1') and (2') follows directly from (1) and (2) of Lemma \ref{bound} respectively. Moreover, (5') follows from the inclusion $z \in W_0^J$ and the equality $(\a')_J=\a_J$ (by Lemma \ref{f5}). By Lemma \ref{bound} (3), we have $\a + w\i(\a) \in z\i(\Phi^-)$. Due to the choice of $\g_0 \in D'$, we have $\a+w\i(\a) \notin D'$, that is, $\a+w\i(\a) \notin \Xi^+(\l, \mu)$. In particular, $\a+w\i(\a) \notin \Xi_1^+(\mu)$ and (3') follows. It remains to verify (4'). Let $\e \in \Phi_J^+$ such that $w\i(\a)-\e, \a+\e \in \Phi$. Then $z(\a+\e)>0$ since $\a \in D_{\max, J}$. Noticing that $z(\a+w\i(\a)) < 0$, we deduce that $z(w\i(\a)-\e)=z(\a+w\i(\a)) - z(\a+\e) <0$. So $w\i(\a)-\e \in D$ and (4') follows.

\subsection{}\label{simply2} Now we assume $D'=\Xi(\l, \mu) \cap z\i(\Phi^-)=\emptyset$. Let $D_1=\{\a \in \SS_0; z(\a)< 0\} \subseteq \SS_0 - J$ and $D_2=\{\b \in \SS_0; \<\mu, \b\>=-1\}$. Then $D_1 \neq \emptyset$ since $z$ is nontrivial. Moreover, $D_1 \cap D_2 \subseteq D' =\emptyset$. Let $\a \in D_1$. We have $\mu+\a^\vee \le \l$ by Lemma \ref{add-simple}. If $\mu+\a^\vee$ is weakly dominant, then $\mu + \a^\vee \preceq \l$ (by Lemma \ref{ind}) and hence $\emptyset \neq D_1 \subseteq D'$, a contradiction. Thus $\mu+\a^\vee$ is not weakly dominant. In particular, by \cite[Lemma 6.5]{N2} and that $\mu$ is weakly dominant, $\mu$ is not dominant and hence $D_2 \neq \emptyset$. Let $\a \in D_1$ and $\b \in D_2$ such that $\dist(\a, \b)=\dist(D_1, D_2)$ (see the end of \S \ref{chase}). Therefore, one of (1), (2) and (3) in Lemma \ref{seq} occurs. By the choices of $\a$ and $\b$, we have $\<\mu, \d_j\> \ge 0$ for $2 \le j \le m-1$ and $z(\xi_1), z(\xi_2), z(\e) > 0$.

If Lemma \ref{seq}(1) occurs, set $\th=\d_m + \dots + \d_1$, $n=m$ and $\eta_j=\d_j$ for $1 \le j \le m$. If Lemma \ref{seq}(2) occurs, set $\th=\d_m + \dots + \d_1 +\xi_1 + \xi_2 + \b$, $n=m+3$, $(\eta_1, \eta_2, \eta_3)=(\b, \xi_1, \xi_2)$ and $\eta_{j+3}=\d_j$ for $1 \le j \le m$. If Lemma \ref{seq}(3) occurs, set $\th=\d_m + \dots + \d_1 + 2\e +\xi_1 + \xi_2 + \b$, $n=m+5$, $(\eta_1, \eta_2, \eta_3, \eta_4, \eta_5)=(\b, \e, \xi_1, \xi_2, \e)$ and $\eta_{j+5}=\d_j$ for $1 \le j \le m$. Then $z(\eta_i) > 0$ for $1 \le i \le n-1$. In all cases, we have $s_{\eta_i} \cdots s_{\eta_{n-1}} (\eta_n)=\eta_n + \cdots + \eta_i$ for $1 \le i \le n$. Moreover, by Lemma \ref{seq} one checks directly that $\mu+\a^\vee \to_+ \mu+\th^\vee$ (see \S \ref{chase}). Combining Lemma \ref{ind} with Lemma \ref{seq}, we deduce that $\mu+\th^\vee \le \l$ and hence $\mu+\th^\vee \preceq \l$. In particular, $z(\th)>0$ since $D'=\emptyset$. Let $1 \le i_0 \le n-1$ be the minimal integer such that $$z(\eta_n + \cdots + \eta_{i_0+1})=z(s_{\eta_{i_0+1}} \cdots s_{\eta_{n-1}}(\eta_n)) < 0.$$ Set $\d=\eta_n+\cdots+\eta_{i_0+1} \in \Phi^+$. By exchanging $\xi_1$ and $\xi_2$ in Lemma \ref{seq} if necessary, we may and do further assume that

\

$(\ast)$ {\it in Lemma \ref{seq} (3) (resp. Lemma \ref{seq} (2)), $z(\d+\eta_{i_0}), z(\d+\eta_{i_0-1}) > 0$ if $i_0=4$ (resp. if $i_0=3$).}

\

Let $\g \in \Phi^+$ be a maximal $W_J$-conjugate of $\d$ such that $\g \ge_J \d$ and $z(\g)<0$ (or equivalently, $z s_\g \leq z$). By Lemma \ref{f3}, we have $z s_\g \in W_0^J$.

For $\phi \in \ZZ\Phi$ (resp. $u \in W_0$) we denote by $\supp(\phi)$ (resp. $\supp(u)$) the unique minimal subset $E \subseteq \SS_0$ such that $\phi \in \ZZ\Phi_E$ (resp. $u \in W_E$). Notice that $\supp(\phi)$ is a connected subset of $\SS_0$ if $\phi$ is a root.
\begin{lem} \label{equal}
We have the following properties:

(a) $\eta_{i_0} \notin \supp(\g-\d)$;

(b) $\<\eta_i^\vee, \g\>=\<\eta_i^\vee, \d\>$ for $1 \le i \le i_0$.

As a consequence, $s_{\eta_i} \cdots s_{\eta_{i_0}}(\g)=\g+\eta_{i_0} + \cdots + \eta_i$ for $1 \le i \le i_0$.
\end{lem}
\begin{proof}
Since $z(\g)<0$, $\d \le_J \g$ and $z(\d+\eta_{i_0})>0$, we have $\eta_{i_0} \notin \supp(\g-\d)$ and (a) follows. Moreover, $\<\eta_{i_0}, \g\> \le \<\eta_{i_0}, \d\>=-1$. Since $\eta_{i_0} \neq -\g$ and $\Phi$ is simply laced, we have

(c) $\<\eta_{i_0}^\vee, \g\>=-1=\<\eta_{i_0}^\vee, \d\>$.

Case(1): $\eta_i \notin \supp(\d)$ for $1 \le i \le i_0$. Then by (a) we have $\eta_{i_0} \notin \supp(\g)$. Thus $\supp(\g)$, which is connected, lies in the connected component of $\SS_0-\{\eta_{i_0}\}$ containing $\supp(\d)$. Noticing that $\{\eta_{i_0}, \dots, \eta_1\}$ is connected and $\eta_{i_0}$ is a neighbor of $\supp(\d)$, we deduce (by the requirement of Case(1)) that either $\eta_i =\eta_{i_0}$ or $\eta_i$ and $\supp(\d)$ are in different connected components of $\SS_0-\{\eta_{i_0}\}$. By (c), we only need to consider the latter case, where we have $\<\eta_i^\vee, \g\> = 0=\<\eta_i^\vee, \d\>$ and (b) is proved.

Case(2): Case(1) fails. So either (2) or (3) in Lemma \ref{seq} occurs. Without loss of generality, we assume the latter case occurs and hence $1 \le i_0 \le 5$. By (c), we can assume further $2 \le i_0 \le 5$. We claim $\g \ge_J \d$ is elementary (i.e. a sum of distinct simple roots) and $\supp(\g-\d) \subseteq \SS_0 - \supp(\d)$ is disconnected to $\eta_i$ (for $1 \le i \le i_0$). This will prove (b). If $i_0 \in \{3, 5\}$, by (a) we have $\eta_{i_0} \notin \supp(\g)$, which means $\supp(\g)$ lies in the connected component of $\SS_0 - \{\eta_{i_0}\}$ and hence is of type $A$. If $i_0=4$, by $(\ast)$ and (a) we have $\xi_1, \xi_2 \notin \supp(\g)$. So $\supp(\g)$ lies in the connected component of $\SS_0-\{\xi_1, \xi_2\}$, which is of type $A$. If $i_0=2$, by (a) the coefficient of $\e=\eta_{i_0}$ in $\g$ is one, which means $\g$ is elementary by Lemma \ref{elementary}.
\end{proof}

We define $z_i=z s_{\eta_i} \cdots s_{\eta_1}$ and $z_i'=z s_\g s_{\eta_i} \cdots s_{\eta_1}$ for $0 \le i \le i_0$.
\begin{lem} \label{mini}
For $1 \le i \le i_0$ we have (a) $z(\eta_i), z s_\g (\eta_i) > 0$ and (b) $z_i, z_i' \in W_0^J$.
\end{lem}
\begin{proof}
If $\eta_i \in J$, then (a) follows from the observation $z, z s_\g \in W_0^J$. Now we assume $\eta_i \notin J$. If $\<\g^\vee, \eta_i\> \ge 0$, we have $z s_\g(\eta_i)=z(\eta_i)-\<\g^\vee, \eta_i\> z(\g) \ge z(\eta_i) > 0$ as desired. If $\<\g^\vee, \eta_i\> \le -1$, by the equality $\<\g^\vee, \eta_i\>=\<\d^\vee, \eta_i\>$ in Lemma \ref{equal}, either $\eta_i=\eta_{i_0}$ or $\eta_i=\eta_{i_0-1}$ in the situation of $(\ast)$. In any case, we always have $z s_\g(\eta_i) = z(\g+\eta_i) \ge z(\d+\eta_i)>0$ and (a) is proved.

Let $\varepsilon \in J$. One checks that $z\i z_i(\varepsilon) > 0$ and $\supp(z\i z_i(\varepsilon))=\supp(s_{\eta_i} \cdots s_{\eta_1}(\varepsilon)) \subseteq H:=J \cup \{\eta_j; 1 \le j \le i_0\}$. By (a), we have $z s_\g(H), z(H) \subseteq \Phi^+$. So $z_i(\varepsilon)=z(z\i z_i(\varepsilon)), z_i'(\varepsilon)= z s_\g (z\i z_i(\varepsilon)) > 0$ and (b) is proved.
\end{proof}

\begin{lem} \label{per}
Let $1 \le i \le i_0$. Then one of the following fails:

(1) $\<z_{i-1}'(\mu), z s_\g(\eta_i)\>=0$;

(2) $z_{i-1}' w z_{i-1}^{\prime -1} (z s_\g(\eta_i)) < 0$;

(3) $z s_\g(\eta_i) + z_{i-1}' w z_{i-1}^{\prime -1} (z s_\g(\eta_i)) \in \Phi$.

Similar result holds if we replace the pair $(z_i', z s_\g)$ with $(z_i, z)$. In particular, $z s_\g(\eta_i)$ and $z(\eta_i)$ are $z_{i-1}' \tw z_{i-1}^{\prime -1}$-permissible and $z_{i-1} \tw z_{i-1}\i$-permissible respectively.
\end{lem}
\begin{proof}
Let $H_i=\{\eta_1, \dots, \eta_i\}$ and $H=J \cup H_{i_0}$. Assume (2) holds. Then $$z_{i-1}' w z_{i-1}^{\prime -1} (z s_\g(\eta_i))=z s_\g s_{\eta_{i-1}} \cdots s_{\eta_1} w s_{\eta_1} \cdots s_{\eta_{i-1}} (\eta_i) < 0.$$ Since $\supp(s_{\eta_{i-1}} \cdots s_{\eta_1} w s_{\eta_1} \cdots s_{\eta_{i-1}} (\eta_i)) \subseteq H$ and $z s_\g \in W_0^H$ (see Lemma \ref{mini}), we have $s_{\eta_{i-1}} \cdots s_{\eta_1} w s_{\eta_1} \cdots s_{\eta_{i-1}} (\eta_i) < 0$. Noticing that $$\b=\eta_1 \le s_{\eta_1} \cdots s_{\eta_{i-1}} (\eta_i) \in \Phi^+ - \Phi_J$$ we deduce that $w s_{\eta_1} \cdots s_{\eta_{i-1}} (\eta_i) > 0$, which together with $$s_{\eta_{i-1}} \cdots s_{\eta_1} w s_{\eta_1} \cdots s_{\eta_{i-1}} (\eta_i) < 0,$$ implies that  $\supp(w s_{\eta_1} \cdots s_{\eta_{i-1}} (\eta_i)) \subseteq H_{i-1}$.

Assume Lemma \ref{seq}(1) occurs. Then $H_i$ is of type $A$. However, the coefficient of $\b$ in $s_{\eta_1} \cdots s_{\eta_{i-1}} (\eta_i) + w s_{\eta_1} \cdots s_{\eta_{i-1}} (\eta_i) \in \ZZ H_i$ is $2$. So (3) fails.

Assume Lemma \ref{seq}(3) (resp. Lemma \ref{seq}(2)) occurs. If $1 \le i \le 6$ (resp. $1 \le i \le 4$), then $\<z_{i-1}'(\mu), z s_\g(\eta_i)\>=\<\mu, \b\>=-1$ and (1) fails. Otherwise, the coefficient of $\b$ in $s_{\eta_1} \cdots s_{\eta_{i-1}} (\eta_i)$ is $2$. Hence the coefficient of $\b$ in $s_{\eta_1} \cdots s_{\eta_{i-1}} (\eta_i) + w s_{\eta_1} \cdots s_{\eta_{i-1}} (\eta_i) \in \ZZ H_i$ is $4$. Therefore, (3) fails since $H_i$ is of type $D$.

The ``In particular" part follows from Lemma \ref{bound}.
\end{proof}

\

Now we are ready to prove Proposition \ref{hyp'} under the assumption $\Xi(\l, \mu) \cap z\i(\Phi^+)=\emptyset$.

By Lemma \ref{mini}, we have $z_i, z_i' \in W_0^J$ and $z(\eta_i), z s_\g(\eta_i)>0$ for $1 \le i \le i_0$. We show that

(a) $z_{i-1} \overset {z(\eta_i)} \longleftrightarrow z_i$ and $z_{i-1}' \overset {z s_\g(\eta_i)} \longleftrightarrow z_i'$ for $1 \le i \le i_0$.

To this end, by Lemma \ref{per} and Corollary \ref{f4}, it remains to show

(a1) $s_{z s_\g(\eta_i)} z_{i-1}' \tw z_{i-1}^{\prime -1}, s_{z(\eta_i)} z_{i-1} \tw z_{i-1}\i \in \Adm(\l)$.

For $1 \le i \le n$, we set $$\th_i=z_{i-1}\i (z(\eta_i))=(z_{i-1}')\i (z s_\g (\eta_i))=s_{\eta_1} \cdots s_{\eta_{i-1}} (\eta_i).$$ By Corollary \ref{f4}, to verify (a1), it suffices to show $\mu+(\th_i)_J^\vee \preceq \l$ for $1 \le i \le i_0$. We only deal with the case when Lemma \ref{seq}(3) occurs, other cases can be handled similarly. First we show $\mu+\th_i^\vee \preceq \l$ for $i \ge 7$, in which case $\th_i=\eta_i+\cdots+\eta_1$. Indeed, if $i=n=m+5$, then $\mu+\th_i^\vee=\mu+\th^\vee \preceq \l$. Assuming $\mu+\th_{i+1}^\vee \preceq \l$ with $i \ge 7$, we show $\mu+\th_i^\vee \preceq \l$. Noticing that $\<\mu+\th_{i+1}^\vee, \eta_{i+1}\> = \<\mu, \eta_{i+1}\>+1 \ge 1$ if $i \ge 7$, we have $\mu+\th_i^\vee = \mu+\th_{i+1}^\vee-\eta_{i+1}^\vee \preceq \mu+\th_{i+1}^\vee \preceq \l$ as desired. If $1 \le i \le 6$, one checks that $\<\mu, \th_i\>=-1$ and hence $\mu+\th_i^\vee \preceq \mu \preceq \l$ as desired. Therefore, (a1) and hence (a) are proved.

Then we show that

(b) $z_{i_0} \overset {-z(\g)} \longleftrightarrow z_{i_0}'$.

Again by Corollary \ref{f4} we need to show

(b1) $s_{z(\g)} z_{i_0} \tw z_{i_0}\i, z_{i_0} \tw z_{i_0}\i s_{z(\g)} \in \Adm(\l)$.

By Lemma \ref{equal}, $$z_{i_0}\i z(\g)=\g+\eta_{i_0}+\cdots+\eta_1=s_{\eta_2} \cdots s_{\eta_{i_0}}(\g) + \b.$$ So $z_{i_0}\i(z(\g))-\th=\g-\d \in \ZZ \Phi_J$ and $z_{i_0}\i z(\g)$ is a $W_J$-conjugate of $\th$ (see Lemma \ref{f5}). Now (b1) follows from Corollary \ref{f4} since $\mu+\th^\vee \preceq \l$.

Set $\vartheta=(s_{\eta_2} \cdots s_{\eta_{i_0}}(\g))_J$. If Lemma \ref{seq}(1) occurs, $\supp(\vartheta) \in J \cup \{\d_m, \dots, \d_2\}$. If Lemma \ref{seq}(2) or (3) occurs, $\d_1+\cdots+\d_m \le \vartheta$ (since $\a=\d_m$ and $\b=\d_1$ are contained in $\supp(\vartheta)$) and the coefficient of $\b$ in $\vartheta$ is one. Then $\<\mu, \vartheta\> \ge 0$ in all cases. Since $w s_{\eta_2} \cdots s_{\eta_{i_0}}(\g) \ge_J \vartheta$, we have $$\<\mu, w s_{\eta_2} \cdots s_{\eta_{i_0}}(\g)\> \ge \<\mu, \vartheta\> \ge 0.$$ Moreover, by Lemma \ref{positive}, $\<\mu, w(\b)\>=\<\mu, w_J(\b)\> \ge 1$ since $w_J(\b) \in \Phi^+ - \Phi_{J_{\nu_{\tw}}}$ is $J$-dominant. Therefore, $$\<\mu, w z_{i_0}\i z(\g)\>=\<\mu, w s_{\eta_2} \cdots s_{\eta_{i_0}}(\g)\> + \<\mu, w(\b)\> \ge 1,$$ that is, $\<z_{i_0}(\mu), z_{i_0} w z_{i_0}\i(-z(\g))\>=-\<\mu, w z_{i_0}\i z(\g)\> \le -1$. By Lemma \ref{bound}, $-z(\g)$ is $z_{i_0} \tw z_{i_0}\i$-permissible and (b) follows.

\

Combining (a) with (b) we deduce $z=z_0 \leftrightarrow z_0'=z s_\g$, which finishes the proof of Proposition \ref{hyp'} (for simply laced Dynkin diagrams).

\section{Non-simply laced Dyndin diagrams except $G_2$} \label{non-simply-laced}
In this section, we show Proposition \ref{hyp'} holds if the connected Dynkin diagram $\SS_0$ of $\Phi$ is non-simply laced but not of type $G_2$. Since $G$ is semisimple, there exists a simply laced Dynkin diagram $\SS_0'$ (of type $A_{2k+1}$ or $D_k$ or $E_6$) equipped with a nontrivial diagram involution $\iota$ such that the extended affine Weyl group $\tW$ associated to $\Phi$ is a subgroup of the $\iota$-fixed point subgroup of $\tW'=P' \rtimes W_0'$. Here $P'$ (resp. $W_0'$) denotes the coweight lattice (resp. Weyl group) of the root system $\Phi'$ of $\SS_0'$. For $\a \in \Phi'$ we set $\underline \a=(\a+\iota(\a)) / 2$. Then $\Phi=\{\underline\a; \a \in \Phi'\}$. Denote by $r_\a=s_{\underline \a} \in W_0$ the corresponding reflection of $\underline \a$. Notice that $\<\a, \iota(\a)\>=0$ if $\a \neq \iota(\a)$.
\begin{lem}
The Bruhat order on $(\tW', \SS_0')$ restricts to the Bruhat order on $(\tW, \SS_0)$.
\end{lem}

For $D \subseteq \Phi'$ we set $D^\iota=\{\d \in D; \iota(\d)=\d\}$.
\begin{lem} \label{exchange}
There are exactly two connected components of $\SS_0' - (\SS_0')^\iota$. Moreover, $\iota$ exchanges these two connected components.
\end{lem}

Let $\l, \tw, \mu, J, w$ be as in Section \ref{proof-main}. Let $J'=\{\a \in \SS_0'; \underline \a \in J\}$. Then $\mu=\iota(\mu)$ is $J'$-dominant, $J'$-minuscule and weakly dominant for $\Phi'$. In particular, $\tw \in \Omega_{J'}$. Fix $1 \neq z \in W_0^J$, that is, $1 \neq z=\iota(z) \in {W_0'}^{J'}$.

Let $\chi=\sum_{\a \in \SS_0'} c_\a \a$ and $\chi'=\sum_{\a \in \SS_0'} c_\a' \a$ with $c_\a, c_\a' \in \ZZ$. We define $\chi' \wedge \chi=\sum_{\a \in \SS_0'} \min\{c_\a, c_\a'\} \a$.
\begin{lem} \label{o2}
Let $\g \in {\Phi'}^+ - \Phi'_{J'}$ such that $\g-\iota(\g) \in \ZZ \Phi'_{J'}$. Then $\g \wedge \iota(\g), \g-\g \wedge \iota(\g) \in {\Phi'}^+ \cup \{0\}$. Moreover, if $\g \neq \iota(\g)$ and $\<\d^\vee, \g\>=\<\d^\vee, \iota(\g)\>=-1$ for some $\d \in \Phi_{J'}'$, then $\d=\iota(\d)$ and $\<\d, \g \wedge \iota(\g)\>=-1$.
\end{lem}
\begin{proof}
We argue by induction on the height of $\g$. If $\g$ is $J'$-antidominant, then so is $\iota(\g)$, which means $\g=\iota(\g)$ since $\g-\iota(\g) \in \ZZ \Phi'_{J'}$. Assume $\g$ is not $J'$-antidominant. Then there exists $\a \in J'$ such that $\g'=\g-\a \in {\Phi'}^+$. Let $\xi'=\g' \wedge \iota(\g')$. If $\g'=\iota(\g')$, the statement follows easily. Now we assume $\g' \neq \iota(\g')$. By induction hypothesis, $\g'=\xi'+\varepsilon'$ for some $\varepsilon' \in \Phi_{J'}^{\prime +}$. Noticing that $\<{\g'}^\vee, \varepsilon'\>=-\<{\g'}^\vee, \a\>=1$ we have $\a \neq \varepsilon'$ and either $\<\a^\vee, \varepsilon'\>=-1$ or $\<\a^\vee, \xi'\>=-1$.

Case(1): $\a=\iota(\a)$. If $\<\a^\vee, \varepsilon'\>=\<\a^\vee, \iota(\varepsilon')\>=-1$, then $\<\a^\vee, \xi'\>=0$ and hence $\<\a^\vee, \xi'+\varepsilon'+\iota(\varepsilon')\>=-2$, which is a contradiction since $\Phi'$ is simply laced. So $\<\a^\vee, \xi'\>=-1$ and $\g \wedge \iota(\g)=\xi'+\a$ as desired.

Case(2): $\a \neq \iota(\a)$. If $\<\a^\vee, \xi'\>=0$, then $\<\a^\vee, \varepsilon'\>=-1$. By definition $\supp(\varepsilon') \subseteq \SS_0'-(\SS_0')^\iota$. Thus $\supp(\a+\varepsilon')$ (which is connected) lies in a connected component of $\SS_0'-(\SS_0')^\iota$. By Lemma \ref{exchange}, we have $\supp(\a+\varepsilon') \cap \supp(\iota(\a+\varepsilon'))=\emptyset$ and hence $\g \wedge \iota(\g)=\xi'$ as desired. Assume $\<\a^\vee, \xi'\>=-1$. Then $\<\a^\vee, \varepsilon'\>=0$. If $\<\a^\vee, \iota(\varepsilon')\>=0$, then $\<\a+\xi'+\varepsilon', \iota(\a+\xi'+\varepsilon')\>=-2$, which is impossible. So $\<\a^\vee, \iota(\varepsilon')\> \ge 1$. Then $\varepsilon'-\iota(\a) \in \Phi_{J'}^{\prime +} \cup \{0\}$ (since $\a$ is a simple root) and $\g=\a+\xi'+\iota(\a)+(\varepsilon'-\iota(\a))$, which means $\g \wedge \iota(\g)=\a+\xi'+\iota(\a)$ as desired.

Finally we show the ``Moreover" part. Set $\xi=\g \wedge \iota(\g) \in {\Phi'}^+ - \Phi_{J'}'$. Since $\g \neq \iota(\g)$, $\g=\xi+\varepsilon$ for some $\varepsilon \in \Phi_{J'}^{\prime +}$ with $\varepsilon \neq \iota(\varepsilon)$. If $\<\d^\vee, \xi\> \ge 1$, then $\<\d^\vee, \varepsilon\>=\<\d^\vee, \iota(\varepsilon)\> \le -2$, which means $-\d=\varepsilon=\iota(\varepsilon)$ (since $\Phi'$ is simply laced), a contradiction. If $\<\d^\vee, \xi\>=0$, then $\<\d^\vee, \varepsilon\>=\<\d^\vee, \iota(\varepsilon)\>=-1$ and $\<\d^\vee, \xi+\varepsilon+\iota(\varepsilon)\>=-2$ which is again impossible. Thus $\<\d^\vee, \xi\>=-1$ and $\<\d^\vee, \varepsilon\>=\<\d^\vee, \iota(\varepsilon)\>=0$. It remains to show $\d =\iota(\d)$. Assume otherwise. Then both $\xi+\d+\iota(\d)$ and $\xi+\varepsilon+\iota(\varepsilon)$ are positive roots and $\<\xi+\d+\iota(\d), \xi+\varepsilon+\iota(\varepsilon)\>=-2$, a contradiction. Therefore, $\d = \iota(\d)$ and the proof is finished.
\end{proof}

For $\phi \in \ZZ\Phi'-\ZZ\Phi_{J'}'$ we define $$\ca_\phi =\{\g \in \Phi'; z(\g)<0, \g-\phi \in \ZZ \Phi'_{J'}\}.$$ If $\ca_\phi \neq \emptyset$, by Lemma \ref{f5}, $\ca_\phi$ contains a unique $J'$-antidominant root, which we denoted by $\vartheta_\phi$. If $\phi-\iota(\phi) \in \ZZ \Phi_{J'}'$, then $\ca_\phi^\iota \neq \emptyset$ if and only if $\ca_\phi \neq \emptyset$.
\begin{lem} \label{o3}
We have the following properties:

(1) If $\ca_\phi^\iota \neq \emptyset$, then $z s_\g, z r_\g \in {W_0'}^{J'}$ for $\g \in (\ca_\phi)_{\max, J'}$ satisfying $\g \wedge \iota(\g) \in (\ca_\phi^\iota)_{\max, J'}$.

(2) If $\ca_{\phi+\iota(\phi)}=\emptyset$, then $z s_\g, z r_\g \in {W_0'}^{J'}$ for $\g \in (\ca_\phi)_{\max, J'}$.
\end{lem}
\begin{proof}
First we have $z s_\g \in {W_0'}^{J'}$ by Lemma \ref{f3}.

(1) Set $\g_0=\g \wedge \iota(\g) \in (\ca_\phi^\iota)_{\max, J'}$. By Lemma \ref{o2}, $\g=\g_0+\varepsilon$ for some $\varepsilon \in \Phi_{J'}^{\prime +} \cup \{0\}$. It remains to show $z r_\g \in {W_0'}^{J'}$. If it is false, then $\g \neq \iota(\g)$ and $z r_\g(\d)<0$ for some $\d \in J'$. Since $z s_\g, z s_{\iota(\g)} \in {W_0'}^{J'}$ and $z(\g), z(\iota(\g)) < 0$, we have $\<\d^\vee, \g\>=\<\d^\vee, \iota(\g)\>=-1$. By Lemma \ref{o2}, $\<\d, \g_0\>=-1$ and $\d=\iota(\d)$. Due to the choice of $\g_0$, we have $zs_\d(\g_0)=z(\g_0+\d)>0$ and $z r_\varepsilon (\g_0)=z(\g_0+\varepsilon+\iota(\varepsilon)) > 0$. Thus \begin{align*}z r_\g(\d)&=z(\g+\d+\iota(\g))=z(\d+2\g_0+\varepsilon+\iota(\varepsilon)) \\ &=z(\d+\g_0)+z(\g_0+\varepsilon+\iota(\varepsilon)) > 0,\end{align*} which is a contradiction.

(2) Assume $\ca_{\phi+\iota(\phi)}=\emptyset$. As in (1) above, if $z r_\g \notin {W_0'}^{J'}$, then $\g \neq \iota(\g)$ and there exists $\d \in \Phi_{J'}^{\prime +}$ such that $zr_\g(\d)=z(\d+\g+\iota(\g))< 0$, that is, $\d+\g+\iota(\g) \in \ca_{\phi+\iota(\phi)}$. This is a contradiction.
\end{proof}

\subsection{} \label{fix}
Define $D=\{\eta \in \SS_0'; z(\eta) < 0\} \subseteq \SS_0' - J'$ and $E=\{\eta \in D; \dist(\eta, \iota(\eta))=d\}$, where $d=\min\{\dist(\eta, \iota(\eta)); \eta \in D\}$. We choose $\a \in E$ such that $\<\mu, \a\> \le \<\mu, \eta\>$ for any $\eta \in E$. Let $\a=\d_1, \d_2, \dots, \d_m=\iota(\a)$ be a geodesic in $\SS_0'$, where $m$ is a positive odd integer. Set $\th=\d_1+\cdots+\d_m$ and $k_0=(m+1)/2$. By the choice of $\a$, we have $z(\d_i) > 0$ for $2 \le i \le m-1$.
\begin{lem} \label{o1}
One of the following cases occurs:

(1) $\a \neq \iota(\a)$ and $\<\mu, \th'\> \ge 0$ with $\th'=\th-\a-\iota(\a)$. Moreover, $\mu+\a^\vee+\iota(\a^\vee) \preceq \l$ (resp. $\mu+\th^\vee \preceq \l$) if $\<\mu, \th'\> \ge 1$ (resp. $\<\mu, \th'\> = 0$);

(2) $\a=\iota(\a)$ and $\mu+\a^\vee \preceq \l$;

(3) $\SS_0'$ is of type $E_6$, $\a=\a_2$, $\<\mu, \a_4\>=-1$ and $\mu+\a^\vee+\a_4^\vee \preceq \l$;

(4) $\SS_0'$ is of type $D_n$, $\{\a, \iota(\a)\} = \{\a_{n-1}, \a_n\}$ with $\<\mu, \a\> \ge 0$, $\<\mu, \a_{n-2}\>=-\<\mu, \a_{n-d}\>=1$ with $3 \le d \le n-2$, $\a_{n-2}, \a_i \in J$, $\<\mu, \a_i\>=0$ for $n-d+1 \le i \le n-3$. Moreover, $\mu+  \a^\vee + \iota(\a^\vee) + \a_{n-2}^\vee + \cdots + \a_{n-d}^\vee \preceq \l$.
\end{lem}
The proof is given in \S \ref{prf-o1}.

\begin{lem} \label{o0}
If $\a \neq \iota(\a)$ and $\ca_{\a+\iota(\a)}=\emptyset$, then $\supp(\g) \subseteq \SS_0-(\SS_0')^\iota$ for $\g \in \ca_\a$. In particular, there exists $u \in W_{J'-(\SS_0')^\iota}$ such that $u=\iota(u)$ and $u(\a)=\g$.
\end{lem}
\begin{proof}
Recall $k_0=(m+1)/2$ and $\d_{k_0}=\iota(\d_{k_0})$. One checks (by the type of $\SS_0'$) that $\iota$ exchanges the two connected components of $\SS_0' - \{\d_{k_0}\}$ which contain $\a$ and $\iota(\a)$ respectively. Thus it suffices to show $\d_{k_0} \notin \supp(\g)$. Assume Otherwise. Then $\g \ge \d_1+\cdots+\d_{k_0}$ (since $\supp(\g)$ is connected) and hence $\d_2, \dots, \d_{k_0} \in J'$. Thus $z(\d_1+\cdots+\d_{k_0})<0$ and $z(\iota(\d_1+\cdots+\d_{k_0-1}))=\iota(z(\d_1+\cdots+\d_{k_0-1})) < 0$. Therefore, $z(\th)=z(\d_1+\cdots+\d_{k_0} + \iota(\d_1+\cdots+\d_{k_0-1})) < 0$, that is, $\th \in \ca_{\a+\iota(\a)} \neq \emptyset$, a contradiction. So $\g=y(\a)$ for some $y \in W_{J'}$ whose support $\supp(y)$ lies in the connected component of $\SS_0'-(\SS_0')^\iota$ containing $\a$. Set $u=y \iota(y) =\iota(y) y$. Then $u=\iota(u)$ and $u(\a)=\g$ as desired.
\end{proof}


\begin{lem} \label{o5}
Let $\g$ be as in Lemma \ref{o3} (1). If $\mu+\g_{J'}^\vee \preceq \l$, then $z \tw r_\g z\i \in \Adm(\l)$. Here $\Adm(\l)$ is defined with respect to $\SS_0$.
\end{lem}
\begin{proof}
If $\g=\iota(\g)$, then $\g \in \Phi$ and $\g_J=\g_{J'}$. By Corollary \ref{f4}, we have $z \tw r_\g z\i = z \tw s_\g z\i \in \Adm(\l)$. Assume $\g \neq \iota(\g)$. Set $\zeta=\g \wedge \iota(\g) \in {\Phi'}^+$ and $\varepsilon=\g-\zeta \in \Phi_{J'}^{\prime +}$ (see Lemma \ref{o2}). Then $z(\zeta+\varepsilon+\iota(\varepsilon)) > 0$ by the maximality of $\g$ with respect to $\le_{J'}$. Since $\iota(\zeta)=\zeta$ and $\iota(\g_{J'})=\g_{J'}$, there exists $u \in W_J$ such that $u\i(\zeta)=\g_{J'}$. Since $u(\mu)$ is weakly dominant, one of the follows three cases occurs:

Case(1): $\<u(\mu), \zeta+\varepsilon+\iota(\varepsilon)\> \ge 1$. Then \begin{align*} t^{z (u(\mu)+\zeta^\vee)} &\geq t^{z (u(\mu)+\zeta^\vee)} t^{z(-\zeta^\vee)} s_{z(\zeta)} =t^{z u(\mu)} s_{z(\zeta)} \\ &\geq t^{z u(\mu)} s_{z(\zeta)} s_{z(\zeta+\varepsilon+\iota(\varepsilon))} = t^{z u(\mu)} s_{z(\varepsilon)} s_{z(\iota(\varepsilon))} s_{z(\g)} s_{z(\iota(\g))} \\ &\geq z \tw r_\g z\i \end{align*} as desired, where the last inequality follows from Lemma \ref{f4} since $\tw \leq_J t^{u(\mu)} r_\varepsilon$.

Case(2): $\<u(\mu), \zeta+\varepsilon+\iota(\varepsilon)\> \le 0$ and $\<u(\mu), \zeta+\varepsilon\>=\<u(\mu), \zeta+\iota(\varepsilon)\> \ge 0$. Then we have \begin{align*} t^{z (u(\mu)+\zeta^\vee)} &\geq t^{z u(\mu)} s_{z(\zeta)} \\ &\geq t^{z u(\mu)} s_{z(\zeta)} s_{z(\varepsilon)} \\ &\geq t^{z u(\mu)} s_{z(\zeta)} s_{z(\varepsilon)} s_{z(\iota(\varepsilon))} \\ &\geq t^{z u(\mu)} s_{z(\zeta)} s_{z(\varepsilon)} s_{z(\iota(\varepsilon))} s_{z(-\zeta)} = t^{z u(\mu)} s_{z(\g)} s_{z(\iota(\g))} \\ &\geq z \tw r_\g z\i \end{align*} as desired.

Case(3): $\<u(\mu), \g\>=-1$ and $\<u(\mu), \varepsilon\>=\<u(\mu), \iota(\varepsilon)\>=0$. Then we have $t^{z u(\mu)} \geq t^{z u(\mu)} s_{z(-\g)} s_{z(-\iota(\g))} \geq z \tw r_\g z\i$ as desired.
\end{proof}

The following lemma is proved in \S \ref{prf-zeta}.
\begin{lem} \label{zeta}
Let $\zeta \in {\Phi'}^+ -\Phi_{J'}'$ be such that $\zeta_{J'}=\iota(\zeta_{J'})$ and $\ca_{2 \zeta} \neq \emptyset$. If $\<\mu, \zeta_{J'}\>=-1$ and $\<\mu, \vartheta_{2 \zeta}\> \ge 0$, then one of the following fails:

(1) $\<\mu, \underline\zeta\>=\<\mu, w(\underline\zeta)\>=0$, that is, $\<\mu, \zeta\>=\<\mu, w(\zeta)\>=0$;

(2) $\<\underline \zeta^\vee, w(\underline \zeta)\>=-1$, that is, $\{\<\zeta^\vee, w(\zeta)\>, \<\zeta^\vee, w\iota(\zeta)\>\}=\{0, -1\}$ if $\iota(\zeta) \neq \zeta$ and $\<\zeta, w(\zeta)\>=-1$ $\iota(\zeta) = \zeta$. In particular, either $\z+w(\z)$ or $\z+w\iota(\z)$ is a root of $\Phi'$.
\end{lem}

Now we are ready to prove Proposition \ref{hyp'}.

\

Suppose Lemma \ref{o1} (1) occurs. We set $\th'=\th-\a-\iota(\a) \in {\Phi'}^+$. Then $\<\a^\vee, \th'\>=\<\a^\vee, \iota(\th')\>=-1$. We choose $\g \in {\Phi'}^+ - \Phi_{J'}'$ such that $z \geq z r_\g \in W_0^J$ and $z \overset {-z(\underline\g)} \longleftrightarrow z r_\g$ as follows.

Case(1): $\ca_{\a+\iota(\a)} \neq \emptyset$. Let $\g \in \ca_{\a+\iota(\a)}$ be as in Lemma \ref{o3} (1). Then $\g_{J'}=\th$ and $\th' \in \Phi_{J'}^{\prime +}$ (since $\supp(\g)$ is connected, which means $\g \ge \th$). Thus $\<\mu+\a^\vee+\iota(\a^\vee), \th'\> =\<\mu, \th'\> -2 \le -1$ (since $\mu$ is $J'$-minuscule) and hence $\mu+\th^\vee \preceq \mu+\a^\vee+\iota(\a^\vee)$. So we always have $\mu+\g_{J'}^\vee \preceq \l$. Then $z \tw r_\g z\i \in \Adm(\l)$ follows from Lemma \ref{o5}.

Case(2): $\ca_{\a+\iota(\a)}=\emptyset$ and $\<\mu, \th'\>=0$. Then $\mu+\th^\vee \preceq \l$. If $\ca_\th \neq \emptyset$, let $\g \in \ca_\th$ be as in Lemma \ref{o3} (1) (with $\phi=\th$). Then $z \tw r_\g z\i \in \Adm(\l)$ by Lemma \ref{o5}. Otherwise, we have $z(\th)>0$. Let $\g \in \ca_\a$ be as in Lemma \ref{o3} (2) (with $\phi=\a$). Then $z r_\g \in {W_0'}^{J'}$ since $\ca_{\a+\iota(\a)}=\emptyset$. Let $u \in W_{J'-(\SS_0')^\iota}$ be as in Lemma \ref{o0} such that $\iota(u)=u$ and $u(\a)=\g$. Noticing that $u(\th') > 0$ (since $\d_{k_0} \in \supp(\th') \cap (\SS_0')^\iota$), $\supp(u(\th')) \subseteq {J'} \cup \{\d_2, \dots, \d_{m-1}\}$ and $z({J'} \cup \{\d_2, \dots, \d_{m-1}\}) \subseteq {\Phi'}^+$, we have $z u(\th') > 0$. Moreover, since $\mu$ is weakly dominant, we have $\<\mu, \th\>=2\<\mu, \a\>+\<\mu, \th'\>=2\<\mu, \a\> \ge -1$, which implies $\<\mu, \a\> \ge 0$. Therefore, \begin{align*} t^{z(u(\mu)+\g^\vee+u({\th'})^\vee+\iota(\g^\vee))} & \geq t^{zu(\mu+{\th'}^\vee)} s_{z(\g)} s_{z(\iota(\g))} \\ & \geq t^{zu({\th'})^\vee} s_{zu(\th')} t^{zu(\mu+{\th'}^\vee)} s_{z(\g)} s_{z(\iota(\g))} = t^{zu(\mu)} s_{zu(\th')} s_{z(\g)} s_{z(\iota(\g))} \\ & \geq t^{zu(\mu)} s_{z(\g)} s_{z(\iota(\g))} \\ &\geq z \tw r_\g z\i, \end{align*} where the third inequality follows from the inequality $z(\g+u(\th')+\iota(\g))=z u(\th) \ge z(\th) > 0$ (since $\th$ is ${J'}$-antidominant).

Case(3): $\ca_{\a+\iota(\a)}=\emptyset$ and $\<\mu, \th'\> \ge 1$. Then $\mu+\a^\vee+\iota(\a^\vee) \preceq \l$. Let $\g \in \ca_\a$ be as in Lemma \ref{o3} (2) and let $u \in W_{J'-{\SS_0'}^\iota}$ be as in Case(2). We have \begin{align*} t^{z (u(\mu)+\g^\vee+\iota(\g^\vee))} & \geq t^{z (u(\mu)+\g^\vee+\iota(\g^\vee))} t^{z(-\g^\vee)} s_{z(\g)} = t^{z (u(\mu)+\iota(\g^\vee))} s_{z(\g)} \\ &\geq t^{z (u(\mu)+\iota(\g^\vee))} s_{z(\g)} t^{z(-\iota(\g^\vee))} s_{z(\iota(\g))} = t^{z u(\mu)} s_{z(\g)} s_{z(\iota(\g))} \\ &\geq z \tw r_\g z\i, \end{align*} where the first and second inequalities follow from the fact that $\mu$ is weakly dominant and that $\<\g^\vee, \iota(\g)\>=0$.

Thanks to Corollary \ref{f4}, it remains to show $-z(\underline \g)$ is $z \tw z\i$-permissible. Assume otherwise. By Lemma \ref{bound} (1), $\<\mu, \g\>=\<\mu, w(\g)\>=0$. If $\<\mu, \g_{J'}\>=0$, then $w(\g) \le_{J'} w\i w(\g)=\g$ (see Lemma \ref{shrink}) and $z w(\underline \g) \le z(\underline \g) < 0$, which contradicts Lemma \ref{bound} (2). So $\<\mu, \g_{J'}\>=-1$. We claim that

(i) $\g \in \ca_\th$ and hence $\g_{J'}=\th$.

Assume otherwise. If $\ca_{\a+\iota(\a)} \neq \emptyset$, then $\g \in \ca_{\a+\iota(\a)}=\ca_\th$, a contradiction. So $\ca_{\a+\iota(\a)} = \emptyset$ and $\g \in \ca_\a$. By Lemma \ref{bound} (3), $z(\underline \g + w(\underline \g)) \in \Phi^-$, which implies either $z(\g+w \iota(\g))$ or $z(\g+w(\g))$ lies in ${\Phi'}^-$. In the former case we have $\g+w \iota(\g) \in \ca_{\a+\iota(\a)}$, a contradiction. So $z(\g+w(\g)) \in {\Phi'}^-$. Since the coefficient of $\a$ in $\g+w(\g) \in \Phi'$ is 2, we see that $\SS_0'$ is of type $E_6$, $\a \in \{\a_3, \a_5\}$ and $\d_{k_0}=\a_4$. If $\d_{k_0} \notin \supp(\g+w(\g))$, then $\supp(\g+w(\g)) \subseteq \SS_0'-\{\d_{k_0}\}$ is of type $A$, which is impossible. So $\d_{k_0} \in \supp(\g+w(\a))$ and $\g+w(\g) \ge_{J'} 2\a+\d_{k_0}$. Then we have $$2z(\a+\d_{k_0}+\iota(\a))= z((2\a + \d_{k_0})+ (2 \iota(\a) + \d_{k_0})) < 0$$ and hence $\a+\d_{k_0}+\iota(\a) \in \ca_{\a+\iota(\a)}$, a contradiction. Therefore, (i) is proved.

Again, since either $z(\g+w \iota(\g))$ or $z(\g+w(\g))$ lies in ${\Phi'}^-$, by (i) we have $\ca_{2\g} \neq \emptyset$, which means $\SS_0'$ is of type $E_6$, $\a \in \{\a_3, \a_5\}$. Since $\<\mu, \g_{J'}\>=-1$, $\a \notin J'$ and $\nu_{\tw}$ is dominant, we have $J'=\SS_0'-\{\a_3, \a_5\}$ and $\mu=\o_1^\vee-\o_3^\vee+\o_4^\vee-\o_5^\vee+\o_6^\vee$. Then $\vartheta_{2\g}=\a_1+\a_2+\a_6+2(\a_3+\a_4+\a_5)$ and $\<\mu, \vartheta_{2 \g}\> = 0$. By Lemma \ref{zeta} and Lemma \ref{bound}, $-z(\underline \g)$ is $z \tw z\i$-permissible, a contradiction.

\

Suppose Lemma \ref{o1} (2) occurs. If $\ca_{2\a} \neq \emptyset$ and $\<\mu, \vartheta_{2\a}\>=-1$, let $\g \in \ca_{2\a}$ be as in Lemma \ref{o3} (1). Otherwise, let $\g \in \ca_\a$ be as in Lemma \ref{o3} (1). We show $z \overset {-z(\underline\g)} \longleftrightarrow z r_\g$. First note that $\mu+\g_{J'}^\vee \preceq \l$. So $z \tw r_\g z\i \in \Adm(\l)$ by Lemma \ref{o5}. It remains to show $-z(\underline \g)$ is $z \tw z\i$-permissible. Assume otherwise. As in the above case, we deduce that $\<\mu, \g_{J'}\>=-1$ and that either $z(\g+w(\g))$ or $z(\g+w\iota(\g))$ lies in ${\Phi'}^-$. If $\g \in \ca_{2\a}$, then the coefficient of $\a$ in $\g+w(\g)$ is 4, which is impossible. If $\g \in \ca_\a$, then $\ca_{2\a} \neq \emptyset$. So $\<\mu, \vartheta_{2\a}\> \ge 0$ by our construction of $\g$. Thanks to Lemma \ref{zeta} and Lemma \ref{bound}, $-z(\underline \g)$ is $z \tw z\i$-permissible, a contradiction.

\

Suppose Lemma \ref{o1} (3) occurs. Then $\mu+\a^\vee+\a_4^\vee \preceq \l$. By the choice of $\a$, we have $z(\a_4)>0$ and hence $z r_{\a_4} \in {W_0'}^{J'}$. If $z(\a+\a_4)>0$, then $z r_\a r_{\a_4}=z s_\a s_{\a_4} \in {W_0'}^{J'}$. Similarly as in \S\ref{simply2}, we have $$z \overset {z(\a_4)} \longleftrightarrow z r_{\a_4} \overset {-z (\a)} \longleftrightarrow z r_\a r_{\a_4} \overset {z s_\a(\a_4)} \longleftrightarrow z r_\a$$ as desired. If $z(\a+\a_4)<0$, let $\g \in \ca_{\a+\a_4}$ be as in Lemma \ref{o3}. We show $z \overset {-z(\g)} \longleftrightarrow z r_\g$. Indeed, thanks to Lemma \ref{f4} and Lemma \ref{o5}, it remains to show $-z(\underline\g)$ is $z \tw z\i$-permissible. Note that the coefficients of $\a$ and $\a_4$ in either $\g+w\iota(\g)$ or $\g+w(\g)$ are always 2 and 2. So $\g+w\iota(\g), \g+w(\g) \notin \Phi'$. Thus $\underline \g +w(\underline \g) \notin \Phi$ and the statement follows from Lemma \ref{bound} (3).

\

Suppose Lemma \ref{o1} (4) occurs. We can assume $\a=\a_{n-1}$. Set $\b=\a_{n-d}$ and $\eta_k=\a_{n-d-1+k}$ for $1 \le i \le d$. Set $\th=\eta_1 + \cdots + \eta_{d-1} + \a_{n-1} + \a_n$. Then $\mu + \th^\vee \preceq \l$. Let $0 \le i_0 \le d-1$ be the minimal integer such that $z(\a_n+\eta_d+\cdots+\eta_{i_0+1}) < 0$.

If $i_0=0$, that is, $\ca_{\a+\iota(\a)+\b} \neq \emptyset$, let $\g \in (\ca_{\a+\iota(\a)+\b})_{\max, J'}$. Then $\iota(\g)=\g$ (since $\g \ge \a_{n-1} + \a_n$) and $z \geq z s_\g=z r_\g \in (W_0')^{J'}$ (see Lemma \ref{f3}). We show that $z \overset {-z(\underline\g)} \longleftrightarrow z s_\g$. Indeed, by Corollary \ref{f4}, it remains to show $-z(\g)$ is $z \tw z\i$-permissible. Assume otherwise. Then $\<\mu, w(\g)\>=0$ (see Lemma \ref{bound} (1)). However, $\g_{J'}=\a_n + \cdots + \a_{n-d}$ and $\<\mu, \g_{J'}\> \ge 0$. Thus $\<\mu, \g_{J'}\>=0$, $w(\g) \le_{J'} w\i w(\g)=\g$ (see Lemma \ref{shrink}) and $z w(\g) \le z(\g) < 0$, contradicting Lemma \ref{bound} (2).

Assume $i_0 \ge 1$. If $i_0\le d-2$, choose $\g \in (\ca_{\a+\iota(\a)})_{\max, J'}$ such that $\g \ge_{J'} \a_n+\eta_d+\cdots+\a_{i_0+1}$. Otherwise, let $\g=\a=\a_{n-1}$. By Lemma \ref{f3}, we have $z r_\g \in W_0^J$. Moreover, $z r_\g (\eta_i)=z(\eta_i) > 0$ if $1 \le i \le i_0-1$, and $z r_\g (\eta_{i_0})=z(\g+\eta_{i_0}) \ge z(\a_n+\eta_d + \cdots + \eta_{i_0}) > 0$ (by the choice of $i_0$). In a word, $z, z r_\g \in (W_0')^{H'}$ with $H'= J' \cup \{\eta_1, \dots, \eta_{i_0}\}$. We show $z \leftrightarrow z r_\g$. Set $z_i=z s_{\eta_i} \cdots s_{\eta_1}$ and $z_i'=z r_{\g} s_{\eta_i} \cdots s_{\eta_1}$ for $0 \le i \le i_0$. Then $z_i, z_i' \in {W_0'}^{J'}$ (see Lemma \ref{mini}). Moreover, for $1 \le i \le i_0$, $(z_{i-1}')\i(z r_\g(\eta_i)) = z_{i-1}\i(z(\eta_i))=\eta_1 + \cdots + \eta_i$ is conjugate to $\eta_1=\a_{n-d}$ under $W_J$. Since $\mu+\a_{n-d}^\vee=s_{\a_{n-d}}(\mu) \preceq \l$, it follows from Corollary \ref{f4} that $s_{z(\eta_i)} z_{i-1} \tw z_{i-1}\i, s_{z r_\g(\eta_i)} z_{i-1}' \tw z_{i-1}^{\prime -1} \in \Adm(\l)$. On the other hand, we have $z_{i-1}\i(z(\eta_i))+wz_{i-1}\i(z(\eta_i)) \in \ZZ\Phi_{H'}'$, in which the coefficient of $\eta_1$ is two. So $z_{i-1}\i(z(\eta_i))+wz_{i-1}\i(z(\eta_i)) \notin \Phi'$ (since $\a_{n-1}, \a_n \notin H'$), that is, $z(\underline {\eta_i}) + z_{i-1} w z_{i-1}\i(z (\underline {\eta_i})) \notin \Phi$ (note that $\underline {\eta_i} =\eta_i$). By Lemma \ref{bound}, $z(\eta_i)$ and $z r_\g(\eta_i)$ are $z_{i-1} \tw z_{i-1}\i$-permissible and $z_{i-1}' \tw z_{i-1}^{\prime -1}$-permissible respectively. Thus $z=z_0 \overset {z(\eta_1)} \longleftrightarrow \cdots \overset {z(\eta_{i_0})} \longleftrightarrow z_{i_0}$ and $z_{i_0}' \overset {zr_\g(\eta_{i_0})} \longleftrightarrow \cdots \overset {zs_\g(\eta_1)} \longleftrightarrow z_0'=z r_\g$. Finally we show $z_{i_0} \overset {-z(\underline\g)} \longleftrightarrow z_{i_0}'$. Note that $z_{i_0}\i z(\g)=\g+\eta_{i_0} + \cdots + \eta_1$. As in \S\ref{simply2}, we have $\<\mu, w z_{i_0}\i z(\g)\> \ge 1$ and hence $-z(\underline\g)$ is $z_{i_0} \tw z_{i_0}\i$-permissible. It remains to show $z_{i_0} \tw z_{i_0}\i r_{z(\g)} \in \Adm(\l)$. If $i_0 \le d-2$, then $\g=\iota(\g)$, and $z_{i_0}\i z(\g)$ is conjugate to $\th$ under $W_J$. Thus the statement follows from Corollary \ref{f4} (since $\mu+\th^\vee \preceq \l$). If $i_0=d-1$, then $\g=\a$. We set $\th''=\eta_1+\cdots+\eta_{d-1}$. Then $\<\mu, \th''\>=0$, $z_{i_0}(\th''), z_{i_0}(\th''+\a)<0$ and $z_{i_0}(\th''+\a+\iota(\a))>0$, one checks that \begin{align*} t^{z_{i_0}(\mu+({\th''}^\vee+\a^\vee)+({\th''}^\vee+\iota(\a^\vee))-{\th''}^\vee)} & \geq t^{z_{i_0}(\mu-{\th''}^\vee)} s_{z_{i_0}(\th''+\a)} s_{z_{i_0}(\th''+\iota(\a))} \\ & \geq t^{z_{i_0}(\mu)} s_{z_{i_0}(\th'')} s_{z_{i_0}(\th''+\a)} s_{z_{i_0}(\th''+\iota(\a))} \\ &\geq t^{z_{i_0}(\mu)} s_{z_{i_0}(\th''+\a)} s_{z_{i_0}(\th''+\iota(\a))} = t^{z_{i_0}(\mu)} s_{z(\a)} s_{z \iota(\a)} \\ &\geq z_{i_0} \tw z_{i_0}\i r_{z(\g)}. \end{align*} The proof is finished.

\section{Type $G_2$ root system} \label{G2}
In this section, we show that Proposition \ref{hyp'} holds if $\SS_0$ is of type $G_2$. Let $\a$ and $\b$ be the unique simple long root and simple short root respectively. Let $\l, \mu, J, \tw, w$ be as in Section \ref{proof-main}. Fix $1 \neq z \in W_0^J$.

\subsection{} Assume $J=\{s_\a\}$. Then $\tw=t^\mu s_\a$, $-z(\b), z(\a) > 0$, $\<\mu, \a\>=1$, $\<\mu, \b\> \ge 0$ and $\mu+\a^\vee+\b^\vee \preceq \l$. Set $\g_1=\b$, $\g_2=\a+3\b$, $\g_3=\a+2\b$, $\g_4=2\a+3\b$ and $\g_5=\a+\b$. For $1 \le j \le i \le 5$, we have $z(\g_i) > 0$ if $z(\g_j)>0$. We denote by $(C_i)$ the case that $z(\g_i)<0$ and $z(\g_{i+1}) > 0$. Assume $(C_i)$ holds, then $z \geq z s_{\g_i} \in W_0^J$. By Lemma \ref{bound} and Corollary \ref{f4}, to show $z \overset {-z(\g_i)} \longleftrightarrow z s_{\g_i}$, it suffices to prove $z \tw s_{\g_i} z\i \in \Adm(\l)$.

Assume $(C_1)$ holds. Then $$t^{z(\mu+\a^\vee+\b^\vee)} \geq t^{z(\mu+\a^\vee+\b^\vee)} t^{z(-\b^\vee)} s_{z(\b)} =t^{z(\mu+ \a^\vee)} s_{z(\b)} \geq z \tw s_{\g_1} z\i. $$

Assume $(C_2)$ holds. Then $$t^{z(\mu+\a^\vee+\b^\vee)} \geq t^{z(\mu)} s_{z(\a+3\b)} \geq z \tw s_{\g_2} z\i.$$

Assume $(C_3)$ holds. Then \begin{align*} t^{z(\mu+\a^\vee+\b^\vee)} &\geq t^{z(\mu)} s_{\a+3\b} \geq t^{z(\mu)} s_{z(\a+3\b)} s_{z(2\a+3\b)} \\ &\geq t^{z(\mu)} s_{z(\a+3\b)} s_{z(2\a+3\b)} s_{z(\a+\b)}=t^{z(\mu)} s_{z(\a+2\b)} \geq z \tw s_{\g_3} z\i. \end{align*}

Assume $(C_4)$ holds. Then \begin{align*} t^{z(\mu+\a^\vee+\b^\vee)} \geq t^{z(\mu)} s_{z(\a+3\b)} \geq t^{z(\mu)} s_{z(\a+3\b)} s_{z(\a)}= t^{z(\mu)} s_{z(\a)} s_{z(2\a+3\b)} \geq z \tw s_{\g_4} z\i.  \end{align*}

Assume $(C_5)$ holds. Then \begin{align*} t^{z(\mu+\a^\vee+\b^\vee)} &\geq t^{z(\mu)} s_{z(\a+3\b)} \geq t^{z(\mu)} s_{z(\a+3\b)} s_{z(\a+2\b)} \\ &\geq t^{z(\mu)} s_{z(\a+3\b)} s_{z(\a+2\b)} s_{z(\a)} =t^{z(\mu)} s_{z(\a+\b)} \geq z \tw s_{\g_5} z\i.  \end{align*}

\subsection{} Assume $J=\{s_\b\}$. Then $\tw=t^\mu s_\b$, $z(\b), -z(\a) > 0$, $\<\mu, \b\>=1$, $\<\mu, \a+\b\> \ge 0$ and $\mu+\a^\vee \preceq \l$. Set $\d_1=\a$, $\d_2=\a+\b$, $\d_3=2\a+3\b$, $\d_4=\a+2\b$ and $\d_5=\a+3\b$. For $1 \le j \le i \le 5$, we have $z(\d_i) > 0$ if $z(\d_j)>0$. We denote by $(C_i')$ the case that $z(\d_i)<0$ and $z(\d_{i+1}) > 0$. Assume $(C_i')$ holds, then $z \geq z s_{\d_i} \in W_0^J$. By Lemma \ref{bound} and Corollary \ref{f4}, to show $z \overset {-z(\d_i)}\longleftrightarrow z s_{\d_i}$, it suffices to show $z \tw s_{\d_i} z\i \in \Adm(\l)$.

Assume $(C_1')$ holds. Then $t^{z(\mu+\a^\vee)} \geq t^{z(\mu)} s_{z(\a)} \geq z \tw s_{\d_1} z\i$.

Assume $(C_2')$ holds. Then \begin{align*} t^{z(\mu+\a^\vee)} \geq t^{z(\mu)} s_{z(\a)} \geq t^{z(\mu)} s_{z(\a)} s_{z(2\a+3\b)} =t^{z(\mu)} s_{z(\b)} s_{z(\a+\b)} \geq z \tw s_{\d_2} z\i. \end{align*}

Assume $(C_3')$ holds. Then \begin{align*} t^{z(\mu+\a^\vee)} \geq t^{z(\mu)} s_{z(\a)} \geq t^{z(\mu)} s_{z(\a)} s_{z(\a+2\b)} =t^{z(\mu)} s_{z(\a)} s_{z(2\a+3\b)} \geq z \tw s_{\d_3} z\i. \end{align*}

Assume $(C_4')$ holds. Then \begin{align*} t^{z(\mu+\a^\vee)} \geq t^{z(\mu)} s_{z(\a)} \geq t^{z(\mu)} s_{z(\a)} s_{z(\a+3\b)} = t^{z(\mu)} s_{z(\b)} s_{z(\a+2\b)} \geq z \tw s_{\d_4} z\i.  \end{align*}

Assume $(C_5')$ holds. Then \begin{align*} t^{z(\mu+\a^\vee)} \geq t^{z(\mu)} s_{z(\a)} \geq  t^{z(\mu)} s_{z(\a)} s_{z(\b)} =t^{z(\mu)} s_{z(\b)} s_{\a+3\b} \geq z \tw s_{\d_5} z\i. \end{align*}

\subsection{} Assume $J=\emptyset$. Then $\tw=t^\mu$, $\mu$ is dominant and $\mu+\a^\vee, \mu+\a^\vee+\b^\vee \preceq \l$. If $z(\a)<0$, then $$t^{z(\mu)} s_{z(\a)}=t^{z(\mu+\a^\vee)} t^{-z(\a^\vee)} s_{z(\a)} \leq t^{z(\mu+\a^\vee)} \in \Adm(\l)$$ and hence $z \overset {-z(\a)} \longleftrightarrow z s_{\a}$. Otherwise, we have $z(\b)<0$. If $z(\a+3\b)<0$, then $$t^{z(\mu)} s_{z(\a+3\b)}=t^{z(\mu+\a^\vee+\b^\vee)} t^{-z((\a+3\b)^\vee)} s_{z(\a+3\b)} \leq t^{z(\mu+\a^\vee+\b^\vee)} \in \Adm(\l)$$ and hence $z \overset {-z(\a+3\b)} \longleftrightarrow z s_{\a+3\b}$. If $z(\a+3\b)>0$, then $$t^{z(\mu)} s_{z(\b)} \leq t^{z(\mu)} s_{z(\b)} t^{z(\a^\vee)} s_{z(\a)} \leq t^{z(\mu+\a^\vee+\b^\vee)} s_{z(\a+3\b)} \leq t^{z(\mu+\a^\vee+\b^\vee)} \in \Adm(\l),$$ and $z \overset {-z(\b)} \longleftrightarrow z s_\b$.

\appendix

\section{}
In the appendix, we prove Lemma \ref{seq}, Lemma \ref{o1}, Lemma \ref{empty} and Lemma \ref{zeta} via a case-by-case analysis on the type of the Dynkin diagram of $\SS_0$ (or $\SS_0'$ in Section \ref{non-simply-laced}).

For $\chi \in Y$ we define $D_\chi^+ =\{\a \in \SS_0; \<\chi, \a\> \ge 1\}$, $D_\chi^0 =\{\a \in \SS_0; \<\chi, \a\>=0\}$ and $D_\chi^- =\{\a \in \SS_0; \<\chi, \a\>=-1\}$. For $\a, \a' \in \SS_0$ (or $\SS_0'$) we denote by $[\a, \a']$ be the subset of simple roots that lie on the geodesic (shortest path) in the Dynkin diagram connecting $\a'$ and $\a$. Set $(\a, \a')=[\a, \a']-\{\a, \a'\}$. Let's consider the following condition:

(c) $\SS_0=D_\chi^+ \cup D_\chi^0 \cup D_\chi^-$, and $(\a, \a') \cap D_\chi^+ \neq \emptyset$ for any $\a \neq \a' \in D_\chi^-$.

It is clear that (c) holds if $\chi$ is weakly dominant.

Again, we adopt the labeling of Dynkin diagrams by positive integers as in \cite{Hum}. For $i \in \ZZ_{\ge 1}$ we denote by $\a_i$ the corresponding simple root. We write $\a_k \trianglelefteq \a_{k'}$ if $k \le k'$.
\begin{lem} \label{criterion}
If (c) holds for $\chi \in Y$, then $\chi$ is weakly dominant if one of the following conditions holds:

(1) $\SS_0$ is of type $A$;

(2) $\SS_0$ is of type $D_n$, $\max D_\chi^+ \ge \max (D_\chi^- - \{\a_{n-1}, \a_n\})$ and $\{\a_{n-1}, \a_n\} \not \subseteq D_\chi^-$.
\end{lem}

\begin{lem} \label{plus}
Let $\tw, J, \mu$ be as in Section \ref{proof-main}. Let $\a, \a' \in \SS_0 - J$ with $\<\mu, \a'\> = -1$. Then we have

(1) $\sum_H \<\mu|_H, \a'\> < -1$, where $H$ ranges over connected components of $J$, and $\mu|_H \in \RR \Phi_H^\vee$ denotes the restriction of $\mu$ to $\RR \Phi_H$;

(2) $(\a, \a') \cap D_\mu^+ \neq \emptyset$ if either (i) $\SS_0$ is of type $D$ and $\a, \a' \in \SS_0 - \{\a_{n-1}, \a_n\}$ or (ii) $\SS_0$ is of type $A$.
\end{lem}
\begin{proof}
Notice that $\a' \in \SS_0 - J_{\nu_{\tw}}$ and $\nu_{\tw}=pr_J(\mu)=\mu-\sum_H \mu|_H$. Therefore, we have $$0 < \<pr_J(\mu), \a'\>=\<\mu, \a'\>-\sum_H \<\mu|_H, \a'\>=-1-\sum_H \<\mu|_H, \a'\>$$ as desired.

Under the condition of (2), one checks that $\sum_{H; H \not \subseteq (\a, \a')} \<\mu|_H, \a'\> \ge -1$. Thus, by (1) there exists a connected component $H' \subseteq (\a, \a')$ of $J$ such that $\<\mu|_{H'}, \a'\> < 0$. So $\emptyset \neq H' \cap D_\mu^+ \subseteq (\a, \a') \cap D_\mu^+$ as desired.
\end{proof}

\begin{lem} \label{plus-simple}
Assume $\SS_0$ is simply laced. Let $\chi \in Y$ and $\a \in \SS_0$ such that (c) holds for $\chi$, and $(\a, \a') \cap D_\chi^- \neq \emptyset$ for any $\a' \in D_\chi^-$. Then (c) holds for $\chi + \a^\vee$.
\end{lem}

\subsection{} \label{prf-seq} We prove Lemma \ref{seq}. First note that (c) holds for $\mu$ since $\mu$ is weakly dominant.

Case(1): $\SS_0$ is of type $A$. We show $\mu+\a^\vee$ is weakly dominant. By Lemma \ref{criterion}, it suffices to show (c) holds for $\mu+\a^\vee$. Since (c) holds for $\mu$, by Lemma \ref{plus-simple}, it remains to show $(\a, \a') \cap D_\mu^+ \neq \emptyset$ for $\a' \in D_\mu^-$, which follows from Lemma \ref{plus} (2).

Case(2): $\SS_0$ is of type $D_n$. Thanks to Lemma \ref{plus} (1), one checks that $\{\a_{n-1}, \a_n\} \not \subseteq D_\mu^-$ and the following two statements hold.

(a1) if one of $\a_{n-1}$ and $\a_n$, say $\a_n$, belongs to $D_\mu^-$, then there exists $k \le n-2$ such that $\<\mu, \a_k\> =1$ and $\a_{k-1}, \dots, \a_{n-1} \in J$;

(a2) if $\{\a_{n-1}, \a_n\} \cap D_\mu^- = \emptyset$, then $\max D_\mu^+ > \max D_\mu^-$.

Case(2.1): (a1) occurs. Then $\a=\a_i$ with $i \le k-2$. We show $\th+\a^\vee$ is weakly dominant. By Lemma \ref{criterion}, it suffices to show (c) holds for $\mu+\a^\vee$. Applying Lemma \ref{plus-simple}, it remains to show $(\a, \a') \cap D_\mu^+ \neq \emptyset$ for each $\a' \in D_\mu^-$. If $\a'=\a_n$, we have $\a_k \in (\a, \a') \cap D_\mu^+$ as desired. Otherwise, $\a'=\a_j$ with $j \le k-2$ and the statement follows from Lemma \ref{plus} (2).

Case(2.2): (a2) occurs. We can assume $\a=\a_i$ with $i \neq n-1$.

Case(2.2.1): $i < \max D_\mu^+$. We show $\mu+\a^\vee$ is weakly dominant. Again Lemma \ref{criterion} (2) holds for $\mu+\a^\vee$. It suffices to show (c) holds for $\mu+\a^\vee$. Since $\{\a_{n-1}, \a_n\} \cap D_\mu^- = \emptyset$, this follows from Lemma \ref{plus-simple}.

Case(2.2.2): $\max D_\mu^+ \le i \le n-2$. We show $\mu+\a^\vee$ is weakly dominant. Indeed, one checks that $\mu+\a^\vee \rightarrowtail_+ \mu' := \mu+\a^\vee + \a_{i+1}^\vee + \cdots + \a_{n-1}^\vee + \a_n^\vee$. By Lemma \ref{ind}, it suffices to show $\mu'$ is weakly dominant. By Lemma \ref{criterion}, it suffices to show (c) holds for $\mu'$, which follows directly by Lemma \ref{plus} (2) and the observation that $D_{\mu'}^- \subseteq \{\a_{i-1}, \a_{n-2}\} \cup D_\mu^-$ and $D_{\mu'} \supseteq (\{\a_n, \a_{n-1}\} \cup D_\mu^+) - \{\a_{i-1}, \a_{n-2}\}$.

Case(2.2.3): $i=n$. Suppose first that there exists $1 \le j \le n-1$ such that $\<\mu, \a_j\>=-1$ and $(\a_j, \a) \cap D_\mu^+ =\emptyset$. By Lemma \ref{plus} (1) and the fact $\a \notin J$, we have $j \le n-2$, $\<\mu, \a_{n-1}\>=1$ and $\a_{j+1}, \cdots, \a_{n-1} \in J$. Thus $\mu+\a^\vee \to_+ \mu':=\mu+\a^\vee+\a_{n-2}^\vee + \cdots + \a_j^\vee$. We show $\mu'$ is weakly dominant (and hence Lemma \ref{seq} (1) holds). Since $\a_n \in D_{\mu'}^+$, by Lemma \ref{criterion} it suffices to show (c) holds for $\mu'$, which can be proved similarly as in Case(2.2.2). Now suppose $(\a', \a) \cap D_\mu^+ \neq \emptyset$ for any $\a' \in D_\mu^-$. We show $\mu+\a^\vee$ is weakly dominant. To this end, we have $\<\mu+\a^\vee, \a_n\> \ge 2$ and (c) holds for $\mu+\a^\vee$ by Lemma \ref{plus} and Lemma \ref{plus-simple}. The statement now follows from Lemma \ref{criterion}.

Case(3): $\SS_0$ is of type $E_n$ with $n \in \{6, 7, 8\}$. Let's consider the forgetful map $f$ which associates to each quintuple $(J_{\nu_{\tw}}, \mu, J, \a, \b)$ in Lemma \ref{seq} the quadruple $(J_{\nu_{\tw}}, \mu|_{\nu_{\tw}}, \a, \b)$. It can be easily verified that the image of $f$ is a finite set. Let $c$ be the fiber of $f$ containing $(J_{\nu_{\tw}}, \mu, J, \a, \b)$. Then there exists a unique quintuple $(J_{\nu_{\tw}}, \mu_c, J, \a, \b)$ in $c$, such that for each $(J_{\nu_{\tw}}, \mu', J, \a, \b)$ in $c$, $\mu'-\mu_c$ is a nonnegative linear combination of fundamental coweights.

By a case by case analysis on finitely many $c$'s, we verified Lemma \ref{seq} holds for each $(J_{\nu_{\tw}}, \mu_c, J, \a, \b)$. Now we show it also holds for $(J_{\nu_{\tw}}, \mu, J, \a, \b)$ in $c$.

Case(3.1): $(J_{\nu_{\tw}}, \mu_c, J, \a, \b)$ falls in the case Lemma \ref{seq} (3). Then so does $(J_{\nu_{\tw}}, \mu, J, \a, \b)$. With $(J_{\nu_{\tw}}, J, \a, \b)$ as in Lemma \ref{seq} (3) and $\mu|_{\nu_{\tw}}=\mu_c|_{\nu_{\tw}}$, it can be easily checked that $\mu=\o_1^\vee - \o_5^\vee + \o_6^\vee + k \o_8^\vee$ for some $k \in \ZZ_{\ge 0}$. Moreover, with some simple calculations, we can deduce that $\mu+\a^\vee$ is not weakly dominant, and that $\sum_{k=2}^{m-1} \<\mu, \d_k\>=1$. Furthermore, since $\mu_c+\d_m^\vee + \cdots + \d_1^\vee + 2\e^\vee +\xi_1^\vee + \xi_2^\vee + \b^\vee$ is weakly dominant, it follows that $\mu+\d_m^\vee + \cdots + \d_1^\vee + 2\e^\vee +\xi_1^\vee + \xi_2^\vee + \b^\vee$ is weakly dominant, where $\e=\a_4$ and $\{\xi_1, \xi_2\}=\{\a_2, \a_3\}$. So $(J_{\nu_{\tw}}, \mu, J, \a, \b)$ falls in the case Lemma \ref{seq} (3).

Case(3.2): $(J_{\nu_{\tw}}, \mu_c, J, \a, \b)$ falls in the case Lemma \ref{seq} (2). Then so does $(J_{\nu_{\tw}}, \mu, J, \a, \b)$. The proof is similar as in Case(3.1).

Case(3.3): $(J_{\nu_{\tw}}, \mu_c, J, \a, \b)$ does not fall in the case Lemma \ref{seq} (2) or (3). If $\mu_c + \a^\vee$ is weakly dominant, then so is $\mu + \a^\vee$ and the proposition follows. Assume $\mu_c + \a^\vee$ is not weakly dominant. Then $\mu_c + \d_m^\vee + \cdots + \d_1^\vee$ is weakly dominant and hence so is $\mu + \d_m^\vee + \cdots + \d_1^\vee$. If $\sum_{i=2}^{m-1}\<\mu, \d\> = 0$, the proposition follows either $\mu + \a^\vee$ is or is not weakly dominant. Otherwise, by $\sum_{i=2}^{m-1}\<\mu, \d\> \ge 1$ and $\<\mu, \d_i\> \ge 0$ for $2 \le i \le m-1$, we have $\mu + \d_m^\vee + \cdots + \d_1^\vee \to_- \mu+\a^\vee$. By Lemma \ref{ind}(4), $\mu+\a^\vee$ is weakly dominant and the proposition follows.

\subsection{} \label{prf-o1} We prove Lemma \ref{o1}. Again notice that $\mu$ is weakly dominant for $\Phi$ and hence for $\Phi'$. Moreover, $\mu + \underline \a^\vee \le \l$ by Lemma \ref{add-simple}.

Case(1): $\SS_0'$ is of type $A_n$ with $n$ odd. Then $\iota(\a_k)=\a_{n+1-k}$ for $1 \le k \le n$. If $\a=\a_{(n+1)/2}$, then $\iota(\a)=\a$ and $\mu+ {\underline \a}^\vee = \mu+\a^\vee \le \l$.  Similarly as in Case(1) of \S\ref{prf-seq}, we have $\mu+\a^\vee \preceq \l$ as desired.

Assume $\a=\a_i$ with $1 \le i \le (n-1)/2$. If $(\a, \iota(\a)) \cap D_\mu^+ = \emptyset$, we show $(\a, \iota(\a)) \cap D_\mu^- =\emptyset$ and hence $(\a, \iota(\a)) \subseteq D_\mu^0$. Indeed, if $(\a, \iota(\a)) \cap D_\mu^- \neq \emptyset$, by Lemma \ref{plus} (2) we have $(\a, \a') \cap D_\mu^+ \neq \emptyset$ for each $\a' \in (\a, \iota(\a)) \cap D_\mu^-$, a contradiction. Thus we have $\mu + \underline\a^\vee = \mu + \a_i^\vee + \a_{n+1-i}^\vee \to_+ \mu':=\mu+ \a_i^\vee + \cdots + \a_{n+1-i}^\vee$. We show $\mu'$ is weakly dominant and hence $\mu' \preceq \l$ (see Lemma \ref{ind}). By Lemma \ref{criterion}, it suffices to show (c) holds for $\mu'$. Notice that $D_{\mu'}^- \subseteq D_\mu^- \cup \{\a_{i-1}, \a_{n+2-i}\}$ and $D_{\mu'}^+ \supseteq (D_\mu^+ \cup \{\a_i, \a_{n+1-i}\})-\{\a_{i-1}, \a_{n+2-i}\}$. The statement now follows from Lemma \ref{plus} (2). If $(\a, \iota(\a)) \cap D_\mu^+ \neq \emptyset$, we show $\mu + {\underline\a}^\vee=\mu+\a^\vee+\iota(\a^\vee)$ is weakly dominant. Again it suffices to verify (c) holds for $\mu+{\underline \a}^\vee$. By Lemma \ref{plus-simple}, it remains to show $(\a, \a') \cap D_\mu^+ \neq \emptyset$ for $ \a' \in D_\mu^- \cup \{\iota(\a)\}$. Indeed, if $\a'=\iota(\a)$, the statement follows from our assumption. Otherwise, the statement follows from Lemma \ref{plus} (2). Moreover, by Lemma \ref{plus} (2), we have $|(\a, \iota(\a)) \cap D_\mu^-| < |(\a, \iota(\a)) \cap D_\mu^+|$ and hence $\sum_{k=i+1}^{n-i} \<\mu, \a_k\> \ge 1$ as desired (see Lemma \ref{o1} (1)).

Case(2): $\SS_0'$ is of type $D_n$. Then $\iota$ fixes $\a_k$ for $1 \le k \le n-2$ and exchanges $\a_{n-1}$ and $\a_n$. Since $J'=\iota(J')$, by Lemma \ref{plus} (1) we have $\a_{n-1}, \a_n \in \SS_0' - D_\mu^-$ and $\max D_\mu^+ \ge \max D_\mu^-$. If $\a=\a_i$ with $1 \le i \le n-2$, then we have $\mu+{\underline \a}^\vee = \mu+\a^\vee \preceq \l$ as in Case(2) of \S\ref{prf-seq}. Now we assume $\a=\a_n$. Since $\a_{n-1}, \a_n \in \SS_0' - J'$, $\<\mu, \a_{n-2}\> \ge 0$ by Lemma \ref{plus} (1).

Case(2.1): there exits $\a_j \in D_\mu^-$ such that $(\a_{j+1}, \a_{n-2}) \cap D_\mu^+ = \emptyset$. By Lemma \ref{plus} (1), we have $[\a_{j+1}, \a_{n-3}] \subseteq D_\mu^0 \cap J$, $\a_{n-2} \in J$ and $\<\mu, \a_{n-2}\>=1$. Thus $\mu + \a^\vee + \iota(\a^\vee) \to_+ \mu':= \mu+ \a^\vee + \iota(\a^\vee) + \a_{n-2}^\vee + \cdots + \a_j^\vee$. We show $\mu' \preceq \l$. By Lemma \ref{ind}, it suffices to show
$\mu'$ is weakly dominant. Since $\a_{n-1}, \a_n \in D_{\mu'}^+$, by Lemma \ref{criterion} it remains to verify (c) for $\mu'$. Notice that $D_{\mu'}^- \subseteq (D_\mu^- \cup \{\a_{j-1}\}) - \{\a_j\}$ and $D_{\mu'}^+ \supseteq (D_\mu^+ \cup \{\a_{n-1}, \a_n\}) -\{\a_{n-2}, \a_{j-1}\}$. The statement follows from \ref{plus} (2).

Case(2.2): $(\a', \a_{n-2}) \cap D_\mu^+ \neq \emptyset$ for any $\a' \in D_\mu^-$. We set $\mu'=\mu+\a^\vee+\iota(\a^\vee)$ if $\<\mu, \a_{n-2}\> \ge 1$ and $\mu'=\mu+\a^\vee+\iota(\a^\vee)+\a_{n-2}^\vee$ otherwise. We need to show $\mu'$ is weakly dominant. Since $\a_{n-1}, \a_n \in D_{\mu'}^+$, by Lemma \ref{criterion}, it suffices to verify (c) holds for $\mu'$, which follows similarly as in Case(2.1).

Case(3): $\SS_0'$ is of type $E_6$. Assume $\a=\iota(\a)=\a_2$. If $\mu$ is dominant or $\<\mu, \a\>=-1$, then $\mu+\a^\vee$ is weakly dominant and hence $\mu+\a^\vee \preceq \l$. Otherwise, by Lemma \ref{plus} and the observation $\mu=\iota(\mu)$, we have $-\<\mu, \a_4\> = \<\mu, \a_3\> =\<\mu, \a_5\>=1$ and $J'=\SS_0' - \{\a_2, \a_4\}$. Then $\mu+\a^\vee \to_+ \mu'=\mu+\a^\vee+\a_4^\vee$ and $\mu'$ is dominant. Therefore, $\mu' \preceq \l$ as desired. Assume $\a=\iota(\a)=\a_4$. Then either $\<\mu, \a\>=1$ or $\mu$ is dominant by Lemma \ref{plus}, which means $\mu+\a^\vee$ is weakly dominant as desired. Assume $\a=\a_3$. By Lemma \ref{plus}, either $\mu$ is dominant or $-1=\<\mu, \a\>=-1=\<\mu, \a_1\>=\<\mu, \a_4\>=1$. Then one checks Lemma \ref{o1} (4) holds. Assume $\a=\a_1$. By Lemma \ref{plus}, $\<\mu, \a_i\> \ge 0$ for $i \in \{1, 2, 3\}$. Moreover, $\<\mu, \a_3\>=\<\mu, \a_2\>=1$ if $\<\mu, \a_4\>=-1$. Then one checks Lemma \ref{o1} (4) holds. Therefore, the proof is finished.

\subsection{} \label{prf-empty} We show Lemma \ref{empty}. We assume the five conditions hold for each element of $D_{\max, J}$, and show this will lead to a contradiction. First we make some reductions.

Choose $\eta \in D$. Let $D''=\{\g \in D; \g-\eta \in \ZZ \Phi_J\}$ and let $\Phi''$ be the root system spanned by $\Phi_J$ and $\eta$. Thanks to \cite{CKV}[Proposition 4.2.11], the set of simple roots of $\Phi''$ is $J \cup \{\eta_J\}$. By definition we have that $D''_{\max, J} \subseteq D_{\max, J}$ and that $\tw$ is short for $\Phi''$. Thus by replacing the pair $(D, \Phi)$ with $(D'', \Phi'')$, we can assume that

(i) $J=\SS_0-\{s_\b\}$ and $D \subseteq \{\g \in \Phi; \g-\b \in \Phi_J\}$ for some $\b \in \SS_0$. Moreover, $\mu$, which is $J$-dominant and $J$-minuscule, is noncentral on each connected component of $J$.

By (3') and (1') (of Lemma \ref{empty}) we also have

(ii) $\<\mu, \b\>=-1$, $\vartheta_{2\b} \in \Phi$ and $\<\mu, \vartheta_{2\b}\> = 0$, where $\vartheta_{2\b}$ denotes the unique $J$-antidominant and $J$-minuscule character in $2\b+\ZZ \Phi_J$.

Let $\g \in D$. Then $\g-\b \in \ZZ \Phi_J$ and we can write $\g=\b+\sum_H {}^H\g$, where $H$ ranges over the connected components of $J$ and ${}^H\g \in \ZZ_{\ge 0} \Phi_H$. By (1') and (ii), there exits at most one connected component $H'$ such that $\<\mu, {}^{H'} \g\>=1$. If this happens, we say $\g$ is of type $H'$.

Assume $\SS_0$ is of type $A$. Then $\a+w(\a) \notin \Phi$ for any $\a \in D_{\max, J}$ since the coefficient $\b$ in $\a+w(\a)$ is two, contradicting (3').

Assume $\SS_0$ is of type $D_n$. There are three cases to consider:

Case(1): $\b \in \{\a_1, \a_{n-1}, \a_n\}$. This is impossible as in the Type $A$ case.

Case(2): $\b=\a_{n-2}$. Then $\mu=\o_k^\vee-\o_{n-2}^\vee+\o_{n-1}^\vee+\o_n^\vee$ (since $\mu$ is noncentral on each connected component of $J=\SS_0 - \{\b\}$). Let $H_1$, $H_{n-1}$ and $H_n$ be the three connected components of $J$ containing $\a_1$, $\a_{n-1}$ and $\a_n$ respectively. If there exists $\a \in D_{\max, J}$ which is of type $H_1$, then ${}^{H_{n-1}} \a={}^{H_n}\a=0$, which implies $w(\a) \ge_J \b+\a_{n-1}+\a_n$ and hence $\<\mu, w(\a)\> \ge 1$, a contradiction to (1'). Thus each root in $D$ is not of type $H_1$. Choose $\a \in D_{\max, J}$ such that ${}^{H_1}\a \not \le_J {}^{H_1}\phi$ for any $\phi \in D$. By symmetry we can assume $\a$ is of type $H_{n-1}$, that is, ${}^{H_{n-1}}\a=\a_{n-1}$. Then ${}^{H_n}((w\i(\a))=\a_n$, ${}^{H_{n-1}}((w\i(\a))=0$ and ${}^{H_1}\a <_J {}^{H_1}((w\i(\a)) \in \Phi_{H_1}^+$. By (4') we have $w\i(\a)-\a_n \in D$. However, ${}^{H_1}(w\i(\a)-\a_n) = {}^{H_1}(w\i(\a)) > {}^{H_1}\a$, contradicting the choice of $\a$.

Case(3): Suppose $\b=\a_i$ with $2 \le i \le n-3$. Let $H_1$ and $H_n$ be the two connected components of $J$ containing $\a_1$ and $\a_n$ respectively. Then $\mu=\o_k^\vee-\o_i^\vee+\o_j^\vee$ for some $1 \le k \le i-1$ and $j \in \{i+1, n-1, n\}$.

Case(3.1): $j=i+1$. If there exists $\a \in D_{\max, J}$ which is of type $H_1$, then ${}^{H_n}\a=0$. Thus $w(\a) \ge_J \b+2\a_{i+1}$ and hence $\<\mu, w(\a)\> \ge 1$, a contradiction. Therefore, each root of $D$ is not of type $H_1$. Choose $\a \in D_{\max, J}$ such that ${}^{H_1}\a \not \le_J {}^{H_1}\phi$ for any $\phi \in D$. By (3'), we have $\a+w\i(\a) \in \Phi$, which implies ${}^{H_n}\a=\a_{i+1} + \cdots + \a_{n-2} + \a_{n-1}$ and ${}^{H_n}(w\i(\a))=\a_{i+1} + \cdots + \a_{n-2} + \a_n$ (up to exchanging $\a_{n-1}$ and $\a_n$). Moreover, ${}^{H_1}(w\i(\a)) > {}^{H_1}\a$. By (4'), $w\i(\a)-\a_n \in D$ with ${}^{H_1}(w\i(\a)-\a_n) = {}^{H_1}(w\i(\a)) > {}^{H_1}\a$, contradicting the choice of $\a$.

Case(3.2): $j=n$. Since $\<pr_J(\mu), \b\> > 0$, we have $i \ge 3$. Moreover, since $\<\mu, \vartheta_{2\b}\>=0$ (see (ii)), we have $k=i-1$. Set $\d_1=\a_{i+1}+ \cdots + \a_{(n-i-1)/2} + 2\a_{(n-i+1)/2} + \cdots + 2 \a_{n-2} + \a_{n-1} + \a_n$ and $\d_2=\a_{i+1}+ \cdots + \a_{(n-i-1)/2}$. For each $\g \in D_{\max, J}$ we have $\g + w\i(\g) \in \Phi$ by (3'), which implies $n-i$ is odd and

(d1) ${}^{H_n}\g=\d_1$ and ${}^{H_n}(w\i(\g))=\d_2$ if $\g$ is of type $H_n$.

(d2) ${}^{H_n}\g=\d_2$ and ${}^{H_n}(w\i(\g))=\d_1$ if $\g$ is of type $H_1$.

Let $\a \in D_{\max, J}$. Assume $\a$ is of type $H_n$. Then ${}^{H_1}\a=0$. By (d1) and (4') we have $w\i(\a)-(\a_{i+1} + \cdots + \a_{(n-i-1)/2}) = \a_1 + \cdots + \a_i \in D$. Let $\a' \in D_{\max, J}$ such that $\a' \ge \a_1 + \cdots + \a_i$. Then ${}^{H_1}(\a')=\a_1 + \cdots + \a_{i-1}$. By (d2) and (4'), we have $w\i(\a')-(\a_{(n-i+1)/2} + \cdots + \a_{n-1})=\a_2 + \cdots + \a_{n-2} + \a_n \in D$. However, $\<\mu, \a_2 + \cdots + \a_{n-2} + \a_n\> =1$, contradicting (1'). Therefore, $\a$ is of type $H_1$ and hence each root of $D$ is not of type $H_n$. By (d2) and (4') we deduce that $\a_n \le_J w\i(\a)-(\a_{(n-i+1)/2} + \cdots + \a_{n-1}) \in D$, which is of type $H_n$, a contradiction.

Assume $\SS_0$ is of type $E_n$ with $n \in \{6, 7, 8\}$. Then there are only finitely many pairs $(\mu, \b)$ satisfying (i) and (ii). The statement is checked case-by-case by a similar strategy in Case (3) \footnote{By a simple reduction, we only need to address the case where each connected component of $J$ if of type $A$ or type $D$.}. Therefore, the proof is finished.

\subsection{} \label{prf-zeta} Finally we show Lemma \ref{zeta}. Let $\SS_0''$ be the connected component of the Dynkin diagram $J' \cup \{\z_{J'}\}$ containing $\z_{J'}$. Set $J''=\SS_0'' - \{\z_{J'}\} \subseteq J'$. By definition, $\z_{J'}=\z_{J''}$. We claim that

(i) The restriction of $\iota$ to $\SS_0''$ is nontrivial unless $\SS_0''$ is of type $A$.

Assume $\iota$ acts trivially on $\SS_0''$. Then $\iota$ fixes each element of $J''$. If $\SS_0'$ is of type $A$ (resp. $E_6$), then $|J''| \le 1$ (resp. $|J''| \le 2$) and $|\SS_0''|\le 2$ (resp. $|\SS_0''|\le 3$), which means $\SS_0''$ is of type $A$. If $\SS_0'$ is of type $D_n$, then $J'' \subseteq \SS_0' -\{\a_{n-1}, \a_n\}$. Moreover, $\supp(\z_{J'}) \subseteq \SS_0' -\{\a_1, \a_{n-1}, \a_n\}$ because $\ca_{2 \z} \neq \emptyset$. Since $\SS_0' -\{\a_{n-1}, \a_n\}$ is of type $A$, one checks that $\SS_0''=J'' \cup \{\z_{J'}\}$ is also of type $A$. So (i) is proved.

Now we assume (1) and (2) of Lemma \ref{empty} hold, and show this will lead to a contradiction.

Case(1): $\SS_0''$ is of type $A$. Then $\ca_{2\z}=\emptyset$, a contradiction.

Case(2): $\SS_0''$ is of type $D_n$. By (i) $\iota$ fixes $\a_k$ for $1 \le k \le n-2$, and exchanges $\a_{n-1}$ and $\a_n$. Since $\ca_{2\z} \neq \emptyset$, $\z_{J'}=\a_i$ with $2 \le i \le n-2$.

Case(2.1): $\z_{J'}=\a_{n-2}$. Then $\mu |_{\SS_0''}=\o_k^\vee + \o_{n-1}^\vee + \o_n^\vee$ for some $1 \le k \le n-3$. If $\z \ge_{J''} \a_k$, then $\supp(\z) \subseteq \SS_0'' - \{\a_{n-1}, \a_n\}$ (since $\<\mu, \z\>=0$). Thus $w(\z) \ge_{J''} \a+\a_{n-1}+\a_n$ and hence $\<\mu, w(\underline \z)\> \ge 1$, contradicting (1) of Lemma \ref{empty}. So exactly one of $\a_{n-1}$ and $\a_n$ lies in $\supp(\z)$. In particular, $\z \neq \iota(\z)$ and hence $\underline \z$ is a short root of the root system $\Phi''$ of $\SS_0''$. However, $\z+w(\z)$ is $\iota$-fixed (since the coefficient of $\b$ is two) and hence $\underline \z + w(\underline \z)$ is a long root of $\Phi''$. This implies $\<\underline \z^\vee, w(\underline \z)\>=0$, contradicting (2) of Lemma \ref{empty}.

Case(2.2): $\z_{J'}=\a_i$ with $2 \le i \le n-3$. Then $\mu |_{\SS_0''}=\o_k^\vee - \o_i^\vee + \o_{i+1}^\vee$ for some $1 \le k \le i-1$. If $\z \ge_{J''} \b+\a_k$, then $\a_{i+1} \notin \supp(\z)$. Thus $w(\z) \ge_{J''} \b+2\a_{i+1}$ and hence $\<\mu, w(\underline \z)\> \ge 1$, a contradiction. So the coefficient of $\a_{i+1}$ in $\z$ is one. By (2), we have either $\z+w(\z) \in \Phi''$ or $\z+\iota w(\z) \in \Phi''$. As in Case(3.1) of \S \ref{prf-empty}, this implies $\z \neq \iota(\z)$, which is impossible as in Case(2.1).

Case(3): $\SS_0''$ is of type $E_6$. By (i) $\iota$ restricts to a nontrivial involution of $\SS_0''$. Assume $\z_{J'} = \a_2$. By Lemma \ref{plus} and the observation $\mu=\iota(\mu)$, we have $\mu|_{\SS_0''}=-\o_2^\vee + \o_4^\vee$. Thus the coefficient of $\a_4$ in $\z$ is one since $\<\mu, \z\>=0$, which implies $\<\mu, w(\z)\>=1$, contradicting (1). Assume $\z_{J'} = \a_4$. By Lemma \ref{plus} and the assumption $\<\mu, \vartheta_{2 \z}\> \ge 0$, we have $\mu |_{\SS_0''}=\o_2^\vee -\o_4^\vee +\o_3^\vee + \o_5^\vee$. Then one checks that either $\<\mu, \z\> \ge 1$ or $\<\mu, w(\z)\> \ge 1$, which contradicts (1). The proof is finished.

\end{document}